\documentclass[letterpaper, 10 pt, conference]{ieeeconf}  

\IEEEoverridecommandlockouts                              

\usepackage{hyperref}       
\usepackage{url}            
\usepackage{booktabs}       
\usepackage{amsfonts}       
\usepackage{nicefrac}       
\usepackage{microtype}      
\usepackage{amsmath}
\usepackage{amssymb}
\usepackage{comment}

\usepackage[mathscr]{euscript}
\usepackage{algorithm}

\usepackage{algpseudocode}
\usepackage{algcompatible}

\newtheorem{theorem}{Theorem}

\newtheorem{lemma}{Lemma}
\newtheorem{remark}{Remark}
\newtheorem{assumption}{Assumption}

\usepackage{array,multirow,graphicx}
 \usepackage{float}
\usepackage{caption}
\usepackage[dvipsnames,table,xcdraw]{xcolor}
\newcommand{\fy}[1]{{\color{black}#1}}
\newcommand{\me}[1]{{\color{black}#1}}

\usepackage[textsize=small]{todonotes}

\usepackage{tcolorbox}

\title{\LARGE \bf
 Distributed Gradient Tracking Methods with Guarantees for Computing a Solution to Stochastic MPECs
}

\author{Mohammadjavad Ebrahimi$^{1}$, Uday V. Shanbhag$^{2}$, and Farzad Yousefian$^{1}$
\thanks{*This work was funded in part by the ONR grants N$00014$-$22$-$1$-$2757$ and N$00014$-$22$-$1$-$2589$, and in part by the DOE grant DE-SC$0023303$.}
\thanks{$^{1}$Ebrahimi and Yousefian are with the Department of Industrial and Systems Engineering, Rutgers, the State University of New Jersey, United States.
        {\tt\small me586@scarletmail.rutgers.edu, farzad.yousefian@rutgers.edu}. $^{2}$Shanbhag is with the Department of Industrial and Manufacturing Engineering, The Pennsylvania State University, United States.
        {\tt\small udaybag@psu.edu}.}%
}

\begin{document}

\maketitle
\thispagestyle{empty}
\pagestyle{empty}

\begin{abstract}
We consider a class of hierarchical multi-agent optimization problems over networks where agents seek to compute an approximate solution to a single-stage stochastic mathematical program with equilibrium constraints (MPEC). MPECs subsume several important problem classes including Stackelberg games, bilevel programs, and traffic equilibrium problems, to name a few. Our goal in this work is to provably resolve stochastic MPECs in distributed regimes where \fy{the agents only have} access to their local objectives \fy{and an inexact best-response to the lower-level equilibrium problem}. To this end, we devise a new method called  {\it {\bf randomized \fy{smoothed} distributed zeroth-order gradient tracking}} (rs-DZGT). This is a novel gradient tracking scheme where agents employ \fy{a} zeroth-order implicit scheme to approximate their (unavailable) local gradients. Leveraging the properties of a randomized smoothing technique, we establish the convergence of the method and derive complexity guarantees for computing a stationary point \fy{of an optimization problem with} a smoothed implicit global objective. We also provide preliminary numerical experiments where we compare the performance of rs-DZGT on networks under different settings with that of its centralized counterpart.

\end{abstract}

\section{INTRODUCTION}
We consider an in-network hierarchical distributed optimization problem among $m$ agents, of the form
\begin{equation}\label{eqn:prob}
 \begin{aligned}
\hbox{minimize}_x&\quad f(x)\ \triangleq \frac{1}{m}\sum_{i=1}^m \mathbb{E}_{\xi_i}\left[\tilde{h}_i(x,z(x),\xi_i)\right]\\
 \hbox{subject to} &\quad  z(x) \in \mbox{SOL}(\mathcal{Z}(x),F(x,\bullet)), 
\end{aligned} 
\end{equation}
where agent $i$ is associated with a stochastic local objective function $\tilde{h}_i:\mathbb{R}^n\times  \mathbb{R}^p \times  \mathbb{R}^d\to \mathbb{R}$ and an independent random vector $\xi_i \in \mathbb{R}^d$. Here, $F:\mathbb{R}^n\times  \mathbb{R}^p  \to \mathbb{R}^p$ is a real-valued mapping and $\mathcal{Z}(x) \subseteq \mathbb{R}^p$ is a \fy{set, parametrized} by the variable $x$. 
Mapping \fy{$z: \mathbb{R}^n \to \mathbb{R}^p$} denotes the (unknown) solution to a parametric variational inequality (VI) problem, denoted by $\mbox{VI}(\mathcal{Z}(x),F(x,\bullet))$ with the solution set denoted by $\mbox{SOL}(\fy{\mathcal{Z}}(x), F(x,\bullet))$. Recall that, given an $x \in \mathbb{R}^n$, $z \in \mathcal{Z}(x)$ solves the aforementioned VI if $F(x,z)^T(\tilde{z}-z) \geq 0$ for all $\tilde z \in \mathcal{Z}(x)$. Throughout, we assume that $F$ is an expectation-valued mapping given as $F(x,z) \triangleq \mathbb{E}[\tilde{F}(x,z,\zeta)]$ where $\zeta \in \mathbb{R}^d$ is a random variable. When $m=1$, problem \eqref{eqn:prob} boils down to \fy{a single-stage} stochastic variant of mathematical programs with equilibrium constraints (MPEC). \fy{The} MPEC is an immensely powerful mathematical formulation that has been applied in addressing important problem classes including hierarchical optimization \fy{and} Stackelberg games, to name a few~\cite{outrata2013nonsmooth,luo1996mathematical}. In deterministic, small-scale, and centralized settings, nonlinear programming approaches, including interior-point schemes~\cite{anitescu2000solving} and sequential quadratic programming~\cite{fletcher2006local} have been developed. Another avenue for resolving MPECs lies in implicit programming approaches~\cite{kovcvara2004optimization,luo1996mathematical}. The aforementioned schemes, however, cannot accommodate the presence of uncertainty and do not scale well with the problem size. This shortcoming motivated the development of methods for addressing stochastic MPECs (SMPECs), including sample-average approximation (SAA)~\cite{shapiro2008stochastic}, among others. However, the SAA problems are often difficult to solve in large-scale settings as the number of constraints grows linearly with the sample size. Addressing this limitation, in our recent work~\cite{cui2104complexity}, we consider SMPECs and develop a class of inexact smoothing-enabled zeroth-order methods (ZSOL) applied to the implicit formulation with a Lipschitz continuous objective. ZSOL provides \fy{amongst the first} convergence rate guarantees for resolving SMPECs in convex and nonconvex cases. 

\noindent {\bf Gap.} Despite this promise, however, ZSOL can only accommodate centralized regimes where the problem information is accessible by a single computing agent and as such, addressing SMPECs in distributed settings has remained open. Accordingly, the main goal in this work is to compute an approximate solution to distributed SMPECs over multi-agent networks where the information \fy{of the objective function} is only locally known by a group of agents that can communicate over an undirected network. \fy{We assume that each agent has access to an inexact best response to the lower-level equilibrium, with a prescribed accuracy.} Recently, gradient tracking (GT) methods have been developed for solving standard distributed (and stochastic) optimization problems in convex~\cite{pu2021distributed,sun2022distributed} and nonconvex cases~\cite{lu2019gnsd}. In particular, GT schemes bridge the gap between centralized and distributed optimization by being equipped with the same speed of convergence as their centralized counterparts. \fy{However, the} existing GT methods mainly address unconstrained optimization problems and cannot accommodate MPECs. 

\noindent {\bf Contributions.} In this work, \fy{we develop} a class of gradient tracking methods for addressing problem~\eqref{eqn:prob} over undirected networks and make the following main contributions. {\bf (i)} We devise a new method called  {\it randomized \fy{smoothed} distributed zeroth-order gradient tracking} (rs-DZGT). This is a novel GT scheme where agents employ an inexact zeroth-order implicit scheme to approximate their (unavailable) local gradients. {\bf (ii)} Leveraging properties of \fy{randomized smoothing techniques}, we establish the convergence of the method and derive complexity guarantees for computing a stationary point \fy{of an optimization problem with} a smoothed implicit global objective. {\bf (iii)} We also provide preliminary numerical experiments where we compare the performance of rs-DZGT on networks under different settings with that of its centralized counterpart, ZSOL.
 
\noindent {\bf Outline of the paper.} The remainder of the paper is organized as follows. After presenting the notation, as follows, and the preliminaries in Section~\ref{sec:prelim}, the outline of the algorithm is presented in Section~\ref{sec:alg}. We then provide the convergence analysis in Section~\ref{sec:conv} and derive the main complexity result in Thm.~\ref{thm:main}. Preliminary numerical results and concluding remarks are presented in Sections~\ref{sec:num} and~\ref{sec:conc}, respectively. 

\noindent {\bf Notation.}
Throughout, we let $f_i(x) \triangleq \mathbb{E}_{\xi_i}\left[\tilde{h}_i(x,z(x),\xi_i)\right] $ and $\tilde{f}_i(x,\xi_i) \triangleq  \tilde{h}_i(x,z(x),\xi_i)    $ denote the local deterministic and stochastic implicit objective functions. We further define 
\begin{align}\label{eqn:notations}
&\mathbf{x} \triangleq  [x_1,x_2,\ldots,x_m]^T, \quad   \mathbf{y}  \triangleq  [y_1,y_2,\ldots,y_m]^T \in \mathbb{R}^{m \times n}, \notag\\
&\bar{x} \triangleq  \tfrac{1}{m}\mathbf{1}^T \mathbf{x}   \in \mathbb{R}^{1 \times n}, \quad  \bar{y} \triangleq  \tfrac{1}{m}\mathbf{1}^T \mathbf{y}    \in \mathbb{R}^{1 \times n},\notag\\
& f(x) \triangleq  \tfrac{1}{m}\textstyle\sum_{i=1}^{m} f_i(x), \quad  \mathbf{f}(\mathbf{x}) \triangleq  \tfrac{1}{m}\textstyle\sum_{i=1}^{m} f_i(x_i),\notag\\
&   \mathbf{f}^\eta(\mathbf{x}) \triangleq  \tfrac{1}{m}\textstyle\sum_{i=1}^{m} f_i^\eta(x_i)  ,\quad   \boldsymbol{\xi} \triangleq [\xi_1,\xi_2,\ldots,\xi_m]^T\in \mathbb{R}^{m \times d}\notag ,\\
& \nabla \mathbf{f}^\eta(\mathbf{x}) \triangleq [\nabla f_1^\eta(x_{1}), \ldots,  \nabla f_m^\eta(x_{m})]^T \in \mathbb{R}^{m \times n},
\end{align}
where $ f_i^\eta$ is $\eta$-smoothed variant of $f_i$ that will be formally defined. We let $\mathbf{1} \in \mathbb{R}^{m \times 1}$ denote the vector whose elements are all one. Given $i \in [m]$, $x,v \in \mathbb{R}^n$, $z_1,z_2 \in \mathbb{R}^p$, and $\xi \in \mathbb{R}^d$, we define a local zeroth-order gradient estimator as 
\begin{align}\label{eqn:ZO_oracle}
\hat{g}_i(x,v,z_1,z_2,\xi) \triangleq \left(\tfrac{n(\tilde h_i(x+v,z_2,\xi)-\tilde h_i(x,z_1,\xi))}{\eta}\right)\frac{v}{\|v\|}.
\end{align}
Lastly, we let $\|\bullet\|$ denote the Frobenius norm of a matrix. 
\section{Preliminaries}\label{sec:prelim}
Consider \fy{$f$ as given} in problem~\eqref{eqn:prob}. A key challenge in addressing MPECs is that the implicit function $f$ is often nondifferentiable \fy{and} nonconvex. Further, an \fy{analytical} expression for $z(x)$ is often unavailable, which in turn, makes both zeroth- and first-order information of $f$ unavailable. Contending with these challenges, we utilize a smoothing technique that finds its root in the work by Steklov~\cite{steklov1907expressions} and has been employed in both convex \cite{lakshmanan2008decentralized,yousefian2012stochastic} and nonconvex regimes \cite{nesterov2017random}. Given a continuous function $h$ and a smoothing parameter $\eta>0$,  $h^{\eta}(x)\triangleq \mathbb{E}_{u\in \mathbb{B}}\left[h(x+\eta u)\right]$ is a smoothed function, where $u$ is a random vector in the unit ball $\mathbb{B}$, defined as $\mathbb{B}\triangleq \{u\in \mathbb{R}^n\mid \|u\|\le 1\}$. Throughout,
we let $\mathbb{S}$ denote the surface of the ball $\mathbb{B}$, in other words, $\mathbb{S}\triangleq \{u\in \mathbb{R}^n\mid \|u\|= 1\}$. 
\begin{lemma}[{\cite[Lemma 1]{cui2104complexity}}]\label{Properties of spherical smoothin} 
Consider \fy{$h^{\eta}$ as defined above}. \fy{Then, the following results hold.}

\noindent (i) The smoothed function $h^{\eta}$ is continuously differentiable and 
$\nabla_x h^{\eta}(x)=\left(\tfrac{n}{\eta}\right)\mathbb{E}_{v\in \eta \mathbb{S}}\left[\left(h(x+v)-h(x)\right)\tfrac{v}{\|v\|}\right].$

\noindent (ii) Suppose $h$ is Lipschitz continuous with parameter $L_0$. For any $x , y$ we have $\|\nabla h^{\eta}(x)-\nabla h^{\eta}(y)\|\le \frac{L_0n}{\eta}\|x-y\|$.

\noindent \fy{(iii) $|h^{\eta}(x) -h(x)| \leq L_0\eta,$ for any $x \in \mathbb{R}^n$.}
\end{lemma}
\section{Algorithm outline}\label{sec:alg}
The outline of the proposed method is presented by Algorithms~\ref{alg:DZGT} and~\ref{alg:lowerlevel}. Algorithm~\ref{alg:DZGT} is an inexact zeroth-order GT method applied on the implicit problem. Here agent $i$ generates iterates $x_{i,k}$ and $y_{i,k}$. Of these, the former is updated in step 6 using a weight matrix $\mathbf{w} \in \mathbb{R}^{m\times m}$ while the latter is a zeroth-order gradient tracker being updated in step 10. The scalar $\gamma>0$ denotes a stepsize parameter. We highlight two key design elements in this method. (i) The local implicit objectives $\tilde{h}_i(\bullet,z(\bullet))$ are generally nondifferentiable \fy{and} nonconvex. To address the nonsmoothness, we employ the randomized smoothing technique providing each agent with a local stochastic zeroth-order gradient. (ii) As mentioned earlier, even zeroth-order information of the local implicit objectives $\tilde{h}_i(\bullet,z(\bullet))$ is unavailable. We address this through step 9 where only inexact evaluations of $z(\bullet)$ are used. To compute this inexact value, denoted by $z_{\varepsilon_k}(\bullet)$, \fy{we} utilize a standard stochastic approximation method, outlined in Algorithm~\ref{alg:lowerlevel}. Notably, the inexactness level in $z_{\varepsilon_k}(\bullet)$ is crucial in the convergence analysis and is rigorously controlled \fy{by} prescribing a termination criterion in Algorithm~\ref{alg:lowerlevel} \fy{given as} $t_k:=\sqrt{k+1}$. Indeed, this criterion will be derived in the analysis in Theorem~\ref{thm:main} to establish \fy{the convergence} result. 
\begin{algorithm}
\caption{randomized \fy{smoothed} distributed zeroth-order gradient tracking (rs-DZGT)}
\label{alg:DZGT}
\begin{algorithmic}[1]
\State {\bf input} a doubly stochastic weight matrix $\mathbf{w}$, stepsize $\gamma$ and smoothing parameter $\eta$, and local random initial points $x_{i,0} \in \mathbb{R}^n$ for all $i \in [m]$
\State For all $i \in [m]$, agent $i$ generates initial random samples $\xi_{i,0}$ and $v_{i,0} \in \eta \mathbb{S}$ 
\State Call Algorithm~\ref{alg:lowerlevel} to get inexact solutions $z_{\varepsilon_0}({x}_{i,0})$ and $z_{\varepsilon_0}({x}_{i,0}+v_{i,0})$ 
\State Use equation~\eqref{eqn:ZO_oracle} to obtain $y_{i,0}:=g_{i,0}^{\eta,\varepsilon_0}$ where 
\begin{align*}
&g_{i,0}^{\eta,\varepsilon_0}\triangleq \hat{g}_i(x_{i,0},v_{i,0},z_{\varepsilon_0}({x}_{i,0}),z_{\varepsilon_0}({x}_{i,0}+v_{i,0}),\xi_{i,0})
\end{align*}
 \FOR {$k = 0,1,2, \ldots$} {\bf in parallel by all agents} 
 \State ${x}_{i,k+1}:=\sum_{j=1}^mW_{ij}x_{j,k}-\gamma y_{i,k}$
\State Generate random samples $\xi_{i,k+1}$ and $v_{i,k+1} \in \eta \mathbb{S}$ 
\State Call Algorithm~\ref{alg:lowerlevel} to get inexact solutions $z_{\varepsilon_{k+1}}({x}_{i,k+1})$ and $z_{\varepsilon_{k+1}}({x}_{i,k+1}+v_{i,k+1})$
\State Use equation~\eqref{eqn:ZO_oracle} to obtain
\begin{align*}
 g_{i,{k+1}}^{\eta,\varepsilon_{k+1}} \triangleq \hat{g}_i(&x_{i,{k+1}},v_{i,{k+1}},z_{\varepsilon_{k+1}}({x}_{i,{k+1}}),\\ &z_{\varepsilon_{k+1}}({x}_{i,{k+1}}+v_{i,{k+1}}),\xi_{i,{k+1}})
\end{align*}
\State $y_{i,k+1} := \sum_{j=1}^mW_{ij}y_{i,k}+g_{i,k+1}^{\eta,\varepsilon_{k+1}}-g_{i,k}^{\eta,\varepsilon_{k}}$
\ENDFOR
\end{algorithmic}
\end{algorithm}

\begin{algorithm}
\caption{Stochastic approximation for lower-level VI}
\label{alg:lowerlevel}
\begin{algorithmic}[1]
\State {\bf input} upper-level iteration index $k$, a vector $\hat{x}_k$, an arbitrary $z_0 \in \mathcal{Z}(\hat x_k)$, scalars $\hat{\gamma} > \frac{1}{2\mu_F} $ and $\Gamma >0$
 \FOR {$t = 0,1, \ldots, t_k:=\sqrt{k+1} $} 
\State Evaluate the stochastic map $\tilde{F}(\hat{x}_k,z_t,\zeta_t)$
\State Do the update $z_{t+1}:= \Pi_{\mathcal{Z}(\hat x_k)}\left[z_t-\hat{\gamma}_t\tilde{F}(\hat{x}_k,z_t,\zeta_t)\right]$
\State Update the stepsize as $\hat{\gamma}_{t+1}:=\frac{\hat\gamma}{t+1+\Gamma}$
\ENDFOR
\end{algorithmic}
\end{algorithm}
We now provide formal statements of the assumptions. 
\begin{assumption}\label{mixxx} 
The mixing matrix $\mathbf{w} \in \mathbb{R}^{m \times m}$ is symmetric and doubly stochastic and $\rho \triangleq |\underline{\lambda}_{max}(\mathbf{w})|<1$, where $\underline{\lambda}_{max}(\mathbf{w})$ denotes the second largest eigenvalue of $\mathbf{w}$.
\end{assumption}
We note that under Assumption~\ref{mixxx}, $\|\mathbf{w}-\frac{1}{m}\mathbf{1}\mathbf{1}^T\|<1$.
\begin{assumption}\label{assum:samples} For each $i \in [m]$, $\{\xi_{i,k}\}$ and $\{v_{i,k}\}$ are both iid \fy{where $v_{i,k} \in \eta\mathbb{S}$}. Also, $\{\xi_{i,k}\}$ and $\{v_{i,k}\}$ are independent. 
\end{assumption}
\me{\begin{assumption}\label{lower bound for the function}
 \fy{Consider $f$ as given in \eqref{eqn:notations}. Suppose $\inf_{x \in \mathbb{R}^n} f(x) > -\infty$.}
\end{assumption}
\begin{remark}
\fy{In view of Lemma~\ref{Properties of spherical smoothin}, $f(x)-L_0\eta \le f^{\eta}(x)$. Under Assumption~\ref{lower bound for the function} it follows $\inf_{x \in \mathbb{R}^n} f^{\eta}(x) > -\infty$.} 
 
\end{remark}}
\begin{assumption}\label{assum:main1}
For any agent $i \in [m]$, $\tilde h_i(x,\bullet,\xi_i)$ is $\tilde L_0(\xi_i)$-Lipschitz continuous for any $\xi_i$, and $\tilde L_0 \triangleq  \max_{i\in [m]}\sqrt{\mathbb{E}[\tilde L_0^2(\xi_i)]}$ is finite. Also, $\tilde h_i(\bullet,z(\bullet),\xi_i)$ is $ L_0(\xi_i)$-Lipschitz continuous for any $\xi_i$, and $ L_0 \triangleq \max_{i\in [m]}\sqrt{\mathbb{E}[ L_0^2(\xi_i)]}$ is finite.
\end{assumption}
\begin{assumption}\label{assum:main2} $F(x,\bullet)$ is a $\mu_F$-strongly monotone and $L_F$-Lipschitz continuous mapping uniformly in $x$. For any $x \in \mathbb{R}^n$, the set $\mathcal{Z}(x)\subseteq \mathbb{R}^p$ is nonempty closed convex. 
\end{assumption}
\begin{remark}
Notably, we do not assume that the implicit objective is differentiable. The Lipschitz continuity of the implicit function in Assumption~\ref{assum:main1} has been studied in \cite{patriksson1999stochastic} and holds under mild conditions.
\end{remark}

\section{Convergence theory}\label{sec:conv}
Throughout, we utilize the definitions of the exact and inexact local stochastic zeroth-order gradient as follows.
\begin{align}
&g_{i,k}^{\eta} \triangleq \hat{g}_i(x_{i,k},v_{i,k},z(x_{i,k}),z(x_{i,k}+v_{i,k}),\xi_{i,k}),\label{eqn:g_eta}\\
&g_{i,k}^{\eta,\varepsilon_k} \triangleq \hat{g}_i(x_{i,{k}},v_{i,{k}},z_{\varepsilon_{k}}({x}_{i,{k}}),z_{\varepsilon_{k}}({x}_{i,{k}}+v_{i,{k}}),\xi_{i,{k}}).\label{eqn:g_eta_eps}
\end{align}
We also use the following error terms for $i \in [m]$ and $k\geq 0$. 
\begin{align*}
&\delta_{i,k}^{\eta}\triangleq g_{i,k}^{\eta}-\nabla_x f_i^\eta(x_{i,k}), \qquad  \boldsymbol{\delta}_{k}^{\eta} \triangleq [\delta_{1,k}^{\eta},\ldots,\delta_{m,k}^{\eta}]^T\\
& \omega_{i,k}^{\eta,\varepsilon_{k}}\triangleq g_{i,k}^{\eta,\varepsilon_{k}} - g_{i,k}^\eta, \qquad \boldsymbol{\omega}_{k}^{\eta,\varepsilon_{k}} \triangleq [\omega_{1,k}^{\eta,\varepsilon_{k}},\ldots,\omega_{m,k}^{\eta,\varepsilon_{k}}]^T.
\end{align*}
Here we use $\delta_{i,k}^{\eta}$ to denote the stochastic error of the exact smoothed local zeroth-order gradient and $\omega_{i,k}^{\eta,\varepsilon_{k}}$ to denote the zeroth-order local gradient estimation error due to the inexact calls to the lower-level oracle (see step \#9 in Alg.~\ref{alg:DZGT}). From the above definitions, the main update rules of Algorithm~\ref{alg:DZGT} can be compactly cast for all $k\geq 0$ as
\begin{align}
\mathbf{x}_{k+1} &:= \mathbf{w}\,\mathbf{x}_k-\gamma \mathbf{y}_k,\label{eqn:compact_alg1}\\
\mathbf{y}_{k+1} &:= \mathbf{w}\, \mathbf{y}_k +  \nabla_{\mathbf{x}}\mathbf{f}^\eta(\mathbf{x}_{k+1})-\nabla_{\mathbf{x}} \mathbf{f}^\eta(\mathbf{x}_{k}) \notag\\ 
&+\boldsymbol{\delta}_{k+1}^{\eta}-\boldsymbol{\delta}_{k}^{\eta}+\boldsymbol{\omega}_{k+1}^{\eta,{\varepsilon_{k+1}}}-\boldsymbol{\omega}_{k}^{\eta,{\varepsilon_{k}}}.\label{eqn:compact_alg2}
\end{align} Next, we define an auxiliary matrix sequence $\{\underline{\mathbf{y}}_k\}$, employing the true gradient of the smoothed implicit objective  
\begin{align}
 \underline{\mathbf{y}}_{k+1} &:= \mathbf{w}\underline{\mathbf{y}}_k +\nabla_{\mathbf{x}} \mathbf{f}^\eta(\mathbf{x}_{k+1})-\nabla_{\mathbf{x}} \mathbf{f}^\eta(\mathbf{x}_{k}),\label{eqn:y_underbar}
\end{align}
for all $k \geq 0$, where $\underline{\mathbf{y}}_{0}:=\nabla_x \mathbf{f}^\eta(\mathbf{x}_{0})$. Using mathematical induction, it follows that the average of $\underline{\mathbf{y}}_k$ tracks the average of the smoothed local objectives at their local iterates, i.e.,
\begin{align}\underline{\bar{\mathbf{y}}}_k\triangleq \tfrac{1}{m}\mathbf{1}^T \underline{\mathbf{y}}_k = \tfrac{1}{m} \textstyle\sum_{i=1}^m \nabla_\mathbf{x} f_i^\eta(\mathbf{x}_{i,k})  .\label{eqn:y_underbar1}
\end{align}
\begin{lemma}\label{lemma:delta_x_in_terms_of_y}
The following statements hold for all $k \geq 0$. 

\noindent (i) $\bar{\mathbf{x}}_{k+1}-\bar{\mathbf{x}}_{k}= -\gamma \underline{\bar{\mathbf{y}}}_k- \tfrac{\gamma}{m}\mathbf{1}^T(\mathbf{y}_k-\mathbf{1}\underline{\bar{\mathbf{y}}}_k).$

\noindent (ii) $\mathbf{1}^T(\underline{\mathbf{y}}_k-\mathbf{1}\underline{\bar{\mathbf{y}}}_k)=0.$
\end{lemma}
\begin{proof}
\noindent (i)  Multiplying both sides of \eqref{eqn:compact_alg1} by the averaging operator $\frac{\gamma}{m}\mathbf{1}^T$, and invoking the column-stochasticity of $\mathbf{w}$, 
\begin{align*}
\bar{\mathbf{x}}_{k+1} &= \bar{\mathbf{x}}_k - \tfrac{\gamma}{m}\mathbf{1}^T\mathbf{y}_k =\bar{\mathbf{x}}_k - \tfrac{\gamma}{m}\mathbf{1}^T(\mathbf{y}_k-\mathbf{1}\underline{\bar{\mathbf{y}}}_k+\mathbf{1}\underline{\bar{\mathbf{y}}}_k)\\
&=\bar{\mathbf{x}}_k -\gamma \underline{\bar{\mathbf{y}}}_k- \tfrac{\gamma}{m}\mathbf{1}^T(\mathbf{y}_k-\mathbf{1}\underline{\bar{\mathbf{y}}}_k). 
\end{align*}

\noindent (ii) We have 
\begin{align*}
\mathbf{1}^T(\underline{{\mathbf{y_k}}}-\mathbf{1}\underline{\bar{\mathbf{y}}}_k)&=\mathbf{1}^T(\underline{{\mathbf{y_k}}}-\tfrac{1}{m}\mathbf{1}\mathbf{1}^T\underline{{\mathbf{y}}}_k)=\mathbf{1}^T(\mathbf{I}-\tfrac{1}{m}\mathbf{1}\mathbf{1}^T)\underline{{\mathbf{y}}}_k\\
&=(\mathbf{1}^T-\tfrac{1}{m}\mathbf{1}^T\mathbf{1}\mathbf{1}^T)\underline{{\mathbf{y}}}_k =(\mathbf{1}^T-\mathbf{1}^T)\underline{{\mathbf{y}}}_k=0.
\end{align*}
\end{proof}
The following preliminary result establishes that the exact (possibly unknown) zeroth-order local gradient is an unbiased stochastic gradient of the smoothed local objective and has a bounded second moment. Throughout, we let the history of Alg.~\ref{alg:DZGT} be defined by $\mathcal{F}_k\triangleq \cup_{i=1}^m\cup_{t=0}^{k-1}\{\xi_{i,k}\}$ for $k\geq 1$ \me{and $\mathcal{F}_0\triangleq \cup_{i=1}^m \{\xi_{i,0},x_{i,0}\}$}. 
\begin{lemma}\label{lemma:g_ik_eta_props}
Let $g_{i,k}^{\eta}$ be given by \eqref{eqn:g_eta} and suppose Assumptions~\ref{assum:samples},~\ref{assum:main1}, and~\ref{assum:main2} hold. Then, the following holds for any $i\in [m]$ and all $k\geq 0$ almost surely. (i) $\mathbb{E}\left[\delta_{i,k}^\eta\mid \mathcal{F}_k\right] = 0.$  (ii) $ \mathbb{E}\left[\|g_{i,k}^\eta\|^2\mid \mathcal{F}_k\right] \leq n^2L_0^2.$
 (iii) $\mathbb{E}\left[\|\delta_{i,k}^\eta\|^2\mid \mathcal{F}_k\right] \leq n^2L_0^2.$
\end{lemma} 
\begin{proof}
\noindent (i) This follows from the definition of $\delta_{i,k}$ and 
\begin{align*}
 &\mathbb{E}\left[g_{i,k}^\eta\mid \mathcal{F}_k\right] = \mathbb{E}\left[ \tfrac{n(\tilde f_i(x_{i,k}+v_{i,k},\xi_{i,k})-\tilde f_i(x_{i,k},\xi_{i,k}))v_{i,k}}{\eta\|v_{i,k}\|} \mid \mathcal{F}_k\right]\\
& {=}  \mathbb{E}\left[\mathbb{E}\left[\tfrac{n(\tilde f_i(x_{i,k}+v_{i,k},\xi_{i,k})-\tilde f_i(x_{i,k},\xi_{i,k}))v_{i,k}}{\eta\|v_{i,k}\|}\mid \mathcal{F}_k\cup\{v_{i,k}\}\right]\right]\\
&{=} \left(\tfrac{n}{\eta}\right)\mathbb{E}\left[\frac{\left( f_i(x_{i,k}+v_{i,k},\xi_{i,k})- f_i(x_{i,k},\xi_{i,k})\right)v_{i,k}}{\|v_{i,k}\|}\mid \mathcal{F}_k\right]\\
& \overset{\mathbb{E}[v_{i,k}\mid \mathcal{F}_k]=0}{=} \left(\tfrac{n}{\eta}\right)\mathbb{E}_{v_{i,k}}\left[  f_i(x_{i,k}+v_{i,k},\xi_{i,k})\tfrac{v_{i,k}}{\|v_{i,k}\|}\mid \mathcal{F}_k\right]\\
& {=}\nabla f_i^\eta(x_{i,k}).
\end{align*}

 \fy{\noindent  (ii) We have}
\begin{align*}
  &\mathbb{E}\left[\|g_{i,k}^\eta\|^2\mid \mathcal{F}_k\right] \\ &= \mathbb{E}\left[ \tfrac{n^2|\tilde f_i(x_{i,k}+v_{i,k},\xi_{i,k})-\tilde f_i(x_{i,k},\xi_{i,k})|^2\|v_{i,k}\|^2}{\eta^2\|v_{i,k}\|^2} \mid \mathcal{F}_k\right]\\
& {=}  \tfrac{n^2}{\eta^2}  \mathbb{E}\left[ |\tilde f_i(x_{i,k}+v_{i,k},\xi_{i,k})-\tilde f_i(x_{i,k},\xi_{i,k})|^2\mid \mathcal{F}_k\right]\\
&  \overset{{\tiny \hbox{Assump.~}}\ref{assum:main1}}{\leq}  \tfrac{n^2}{\eta^2}  \mathbb{E}\left[ L_0^2(\xi_{i,k})\|v_{i,k}\|^2\mid \mathcal{F}_k\right]  \overset{\|v_{i,k}\|=\eta}{=}  n^2L_0^2.
\end{align*}

\noindent (iii) This result follows from the following relation.
\begin{align*}
 &\mathbb{E}\left[\|\delta_{i,k}^{\eta} \|^2\mid \mathcal{F}_k\right]  = \mathbb{E}\left[\|g_{i,k}^\eta- \nabla f_i^\eta(x_k)\|^2\mid \mathcal{F}_k\right] \\
 &= \mathbb{E}\left[\|g_{i,k}^\eta\|^2 +\|\nabla f_i^\eta(x_k)\|^2 -2 {g_{i,k}^\eta}^T\nabla f_i^\eta(x_k)\mid \mathcal{F}_k\right] \\
& \overset{(i)}{=} \mathbb{E}\left[\|g_{i,k}^\eta\|^2\mid \mathcal{F}_k\right]-\|\nabla f_i^\eta(x_k)\|^2\\
& \leq \mathbb{E}\left[\|g_{i,k}^\eta\|^2\mid \mathcal{F}_k\right]  \overset{(ii)}{\leq} n^2L_0^2.
\end{align*}
\end{proof}

\begin{lemma}\label{lemm:omega_vareps} 
Let Assumptions~\ref{assum:samples},~\ref{assum:main1}, and~\ref{assum:main2} hold and $\mathbb{E}[\|z_{\varepsilon_k}(\bullet) -z(\bullet)\|^2 \mid \mathcal{F}_k]\leq \varepsilon_k$ hold for all $k \geq 0$ almost surely. Then, the following holds for all $i\in [m]$ and $k\geq 0$. 

\noindent (i) $\mathbb{E}[\|\omega_{i,k}^{\eta,\varepsilon_k}\|^2 |\mathcal{F}_k]\le \left(\tfrac{4\tilde L_0^2n^2\varepsilon_k}{\eta^2}\right).$

\noindent (ii)  Suppose that $\{\varepsilon_k\}$ is nonincreasing. Then, \me{for any $k\ge 0$,}
$$\mathbb{E}[\|\mathbf{y}_k-\underline {\mathbf {y}}_k\|^2 ] \le 2m(n^2L_0^2+ \tfrac{4\tilde L_0^2n^2\varepsilon_0}{\eta^2})\left(1 + \tfrac{8(1+\me{\rho^2})}{ (1-\me{\rho^2})^2}\right) .$$  
\end{lemma}
\begin{proof} (i) Using the definition of $\omega_{i,k}^{\eta,\varepsilon_k} $, we have
\begin{align*}
\|\omega_{i,k}^{\eta,\varepsilon_k}\| &= \|g_{i,k}^{\eta,\varepsilon_k}-g_{i,k}^{\eta}\|\\
& = \|\hat{g}_i(x_{i,{k}},v_{i,{k}},z_{\varepsilon_{k}}({x}_{i,{k}}),z_{\varepsilon_{k}}({x}_{i,{k}}+v_{i,{k}}),\xi_{i,{k}})\\
&-\hat{g}_i(x_{i,k},v_{i,k},z(x_{i,k}),z(x_{i,k}+v_{i,k}),\xi_{i,k})\|\\
& \leq (\tfrac{n}{\eta})|\tilde h_i(x_{i,k}+v_{i,k},z_{{\varepsilon}_k}({x}_{i,k}+v_{i,k}),\xi_{i,k})\\
&-\tilde h_i(x_{i,k}+v_{i,k},z({x}_{i,k}+v_{i,k}),\xi_{i,k})| \\
&+(\tfrac{n}{\eta})|\tilde h_i(x_{i,k},z_{\varepsilon_k}({x}_{i,k}),\xi_{i,k})\\
&-\tilde h_i(x_{i,k},z({x}_{i,k}),\xi_{i,k})|.
\end{align*}
Invoking the Lipschitz continuity of $\tilde h_i(x,\bullet,\xi_i)$,  we obtain 
\begin{align*}
\|\omega_{i,k}^{\eta,\varepsilon_k}\|  &\le  (\tfrac{n}{\eta} )\tilde{L}_0(\xi_{i,k})\left\|z_{{\varepsilon_k}}({x}_{i,k}+v_{i,k})-z({x}_{i,k}+v_{i,k})\right\|\\
&+  (\tfrac{n}{\eta} )\tilde{L}_0(\xi_{i,k})\left\|z_{{\varepsilon_k}}({x}_{i,k})-z({x}_{i,k})\right\|.
\end{align*}
From the preceding inequality,  we obtain
\begin{align*}
&\mathbb{E}[\|\omega_{i,k}^{\eta,\varepsilon_k}\| ^2 \mid\mathcal{F}_k] \le 2(\tfrac{n}{\eta} )^2\mathbb{E}[\tilde{L}_0(\xi_{i,k})^2\left\|z_{{\varepsilon_k}}({x}_{i,k}+v_{i,k})\right.\\
&\left.-z({x}_{i,k}+v_{i,k})\right\|^2\mid \mathcal{F}_k]\\
&+ 2(\tfrac{n}{\eta} )^2\mathbb{E}[\tilde{L}_0(\xi_{i,k})^2\left\|z_{{\varepsilon_k}}({x}_{i,k})-z({x}_{i,k})\right\|^2 \mid\mathcal{F}_k].
\end{align*}
Invoking the independence of $\xi_{i,k}$ and $v_{i,k}$, the definition of $\tilde{L_0}$, and the inexactness bound, we obtain the result. 

\noindent (ii) From equations \eqref{eqn:compact_alg2} and \eqref{eqn:y_underbar}, for any $\theta>0$ we have 
\begin{align*}
&\|\mathbf{y}_{k+1}-\underline {\mathbf {y}}_{k+1}\|^2 \\
& = \|\mathbf{w}(\mathbf{y}_{k}-\underline {\mathbf {y}}_{k})+\boldsymbol{\delta}_{k+1}^{\eta}-\boldsymbol{\delta}_{k}^{\eta}+\boldsymbol{\omega}_{k+1}^{\eta,{\varepsilon_{k+1}}}-\boldsymbol{\omega}_{k}^{\eta,{\varepsilon_{k}}}\|^2\\
& \leq (1+\theta)\|\mathbf{w}(\mathbf{y}_{k}-\underline {\mathbf {y}}_{k})\|^2 \\
&+ (1+\tfrac{1}{\theta})\|\boldsymbol{\delta}_{k+1}^{\eta}-\boldsymbol{\delta}_{k}^{\eta} +\boldsymbol{\omega}_{k+1}^{\eta,{\varepsilon_{k+1}}}-\boldsymbol{\omega}_{k}^{\eta,{\varepsilon_{k}}}\|^2.
\end{align*}
Taking conditional expectations on both sides, we obtain 
\begin{align*}
&\mathbb{E}[\|\mathbf{y}_{k+1}-\underline {\mathbf {y}}_{k+1}\|^2\mid \mathcal{F}_k]  \leq (1+\theta)\me{\rho^2}\|\mathbf{y}_{k}-\underline {\mathbf {y}}_{k}\|^2 \\
&+ 4(1+\tfrac{1}{\theta})\mathbb{E}[\|\boldsymbol{\delta}_{k+1}^{\eta}\|^2\mid \mathcal{F}_k] +4(1+\tfrac{1}{\theta})\mathbb{E}[\|\boldsymbol{\delta}_{k}^{\eta}\|^2\mid \mathcal{F}_k] \\
& +4(1+\tfrac{1}{\theta})\mathbb{E}[\|\boldsymbol{\omega}_{k+1}^{\eta,{\varepsilon_{k+1}}}\|^2\mid \mathcal{F}_k] +4(1+\tfrac{1}{\theta})\mathbb{E}[\|\boldsymbol{\omega}_{k}^{\eta,{\varepsilon_{k}}}\|^2\mid \mathcal{F}_k] .
\end{align*}
Invoking the law of total expectations, we have 
\begin{align*}
\mathbb{E}[\|\boldsymbol{\delta}_{k+1}^{\eta}\|^2\mid \mathcal{F}_k] 
=\mathbb{E}[\mathbb{E}[\|\boldsymbol{\delta}_{k+1}^{\eta}\|^2\mid \mathcal{F}_{k+1}]]. 
\end{align*}
Thus, applying Lemmas~\ref{lemma:g_ik_eta_props} and~\ref{lemm:omega_vareps}, we obtain 
\begin{align*}
\mathbb{E}[\|\mathbf{y}_{k+1}-\underline {\mathbf {y}}_{k+1}\|^2\mid \mathcal{F}_k]  &\leq (1+\theta)\me{\rho^2}\|\mathbf{y}_{k}-\underline {\mathbf {y}}_{k}\|^2 \\
&+  8(1+\tfrac{1}{\theta})m \left(n^2L_0^2+ \tfrac{4\tilde L_0^2n^2\varepsilon_k}{\eta^2} \right) .
\end{align*}
Let us set $\theta:=\frac{1-\me{\rho^2}}{2\me{\rho^2}}$. Thus, $(1+\theta)\me{\rho^2} =\frac{1+\me{\rho^2}}{2}<1$. Let us define the following terms $\hat{\rho}:=(1+\theta)\me{\rho^2}$ and  $\hat{b}:=8(1+\tfrac{1}{\theta})m \left(n^2L_0^2+ \tfrac{4\tilde L_0^2n^2\varepsilon_0}{\eta^2} \right)$. Taking expectations from the preceding inequality, we obtain for $k\geq 0$
\begin{align*}
\mathbb{E}[\|\mathbf{y}_{k+1}-\underline {\mathbf {y}}_{k+1}\|^2]  &\leq \hat{\rho}\, \mathbb{E}[\|\mathbf{y}_{k}-\underline {\mathbf {y}}_{k}\|^2] + \hat{b},
\end{align*}
where we use the nonincreasing assumption on $\{\varepsilon_k\}$. Unrolling the preceding relation recursively yields, for $K\geq 1$,
\begin{align*}
\mathbb{E}[\|\mathbf{y}_{K}-\underline {\mathbf {y}}_{K}\|^2]  &\leq \hat{\rho}^K \, \mathbb{E}[\|\mathbf{y}_{0}-\underline {\mathbf {y}}_{0}\|^2] + \hat{b}\sum_{k=0}^{K-1}\hat{\rho}^k \\
&\leq \mathbb{E}[\|\mathbf{y}_{0}-\underline {\mathbf {y}}_{0}\|^2] + \tfrac{\hat{b}}{1-\hat{\rho}}.
\end{align*}
Note that from the definitions of $\mathbf{y}_{k}$ and $\underline {\mathbf {y}}_{k}$, we have 
\begin{align*}
\mathbb{E}[\|\mathbf{y}_{0}-\underline {\mathbf {y}}_{0}\|^2] = \mathbb{E}[\|\boldsymbol{\delta}_{0}+\boldsymbol{\omega}_{0}\|^2] \leq 2m(n^2L_0^2+ \tfrac{4\tilde L_0^2n^2\varepsilon_0}{\eta^2}),
\end{align*}
\fy{where we used} Lemmas~\ref{lemma:g_ik_eta_props} and~\ref{lemm:omega_vareps}. From the last two relations, we \fy{obtain}
\begin{align*}
\mathbb{E}[\|\mathbf{y}_{K}-\underline {\mathbf {y}}_{K}\|^2]  &\leq 2m(n^2L_0^2+ \tfrac{4\tilde L_0^2n^2\varepsilon_0}{\eta^2}) + \tfrac{\hat{b}}{1-\hat{\rho}} \\
& \leq 2m(n^2L_0^2+ \tfrac{4\tilde L_0^2n^2\varepsilon_0}{\eta^2}) + \tfrac{2\hat{b}}{ 1-\underline{\lambda}_{max}^2} \\
&\leq  2m(n^2L_0^2+ \tfrac{4\tilde L_0^2n^2\varepsilon_0}{\eta^2})\left(1 + \tfrac{8(1+\me{\rho^2})}{ (1-\me{\rho^2})^2}\right) .
\end{align*}
\end{proof} 
\begin{remark}
Notably, the bound in Lemma \ref{lemm:omega_vareps} (ii) holds uniformly invariant of $k$ implying that the error due to the inexact calls to the lower level oracle does not lead to the divergence of $\mathbb{E}[\|\mathbf{y}_k-\underline {\mathbf {y}}_k\|^2 ] $ as $k$ grows. Also, the bound in Lemma \ref{lemm:omega_vareps} (ii) goes to zero as the inexactness level reduces.
\end{remark}
\fy{\begin{lemma}\label{lem:bound_bar_y_bar_squared} Let $\{\mathbf{x}_k\}$ be generated by Algorithm $\ref{alg:DZGT}$. \fy{We have, for $k\geq 0,$} 
\begin{align}\label{main-parameter}
\mathbb{E}[\| \underline{\bar{\mathbf{y}}}_k\|^2|\mathcal{F}_k]&\le\tfrac{2L_0n}{m\eta}\mathbb{E}\left[\left\| \mathbf{x}_{k}-\mathbf{1}\bar{{\mathbf x}}_k\right\|^2|\mathcal{F}_k\right]\notag\\
&+2\mathbb{E}[\|\nabla f^\eta(\bar{\mathbf{x}}_k)\|^2|\mathcal{F}_k]\\
\hbox{\fy{and} } \me{\mathbb{E}[\| \underline{\bar{\mathbf{y}}}_k\|^2|\mathcal{F}_k]}&\me{\ge\tfrac{1}{2}\mathbb{E}[\|\nabla f^\eta(\bar{\mathbf{x}}_k)\|^2|\mathcal{F}_k]}\notag\\
 &\me{-\tfrac{L_0n}{m\eta}\mathbb{E}\left[\left\| \mathbf{x}_{k}-\mathbf{1}\bar{{\mathbf x}}_k\right\|^2|\mathcal{F}_k\right].}
\end{align}
\end{lemma}}
\begin{proof} \fy{To show the first inequality, we} may write
\begin{align*}
&\mathbb{E}[\| \underline{\bar{\mathbf{y}}}_k\|^2|\mathcal{F}_k]=\mathbb{E}[\| \underline{\bar{\mathbf{y}}}_k-\nabla f^\eta(\bar{\mathbf{x}}_k)+\nabla f^\eta(\bar{\mathbf{x}}_k)\|^2|\mathcal{F}_k]\\
&{\le} 2\mathbb{E}[\| \underline{\bar{\mathbf{y}}}_k-\nabla f^\eta(\bar{\mathbf{x}}_k)\|^2|\mathcal{F}_k]+2\mathbb{E}[\|\nabla f^\eta(\bar{\mathbf{x}}_k)\|^2|\mathcal{F}_k]\\
&{=} 2\mathbb{E}\left[\left\| \tfrac{1}{m} \textstyle \sum_{i=1}^m \nabla_\mathbf{x} f_i^\eta({x}_{i,k})-\tfrac{1}{m}\textstyle \sum_{i=1}^m \nabla_\mathbf{x} f_i^\eta(\bar{{x}}_k) \right\|^2|\mathcal{F}_k\right]\\
&+2\mathbb{E}[\|\nabla f^\eta(\bar{\mathbf{x}}_k)\|^2|\mathcal{F}_k]\\
 &\le  \tfrac{2}{m} \textstyle \sum_{i=1}^m\mathbb{E}\left[\left\| \nabla_\mathbf{x} f_i^\eta({x}_{i,k})- \nabla_\mathbf{x} f_i^\eta(\bar{{x}}_k)\right\|^2|\mathcal{F}_k\right]\\
&+2\mathbb{E}[\|\nabla f^\eta(\bar{\mathbf{x}}_k)\|^2|\mathcal{F}_k]\\
&\overset{\scriptsize{\mbox{Lemma }} \ref{Properties of spherical smoothin}}{\le}
\tfrac{2L_0n}{m\eta}\mathbb{E}\left[\left\| \mathbf{x}_{k}-\mathbf{1}\bar{{\mathbf x}}_k\right\|^2|\mathcal{F}_k\right]+2\mathbb{E}[\|\nabla f^\eta(\bar{\mathbf{x}}_k)\|^2|\mathcal{F}_k].
\end{align*}
 \fy{To show the second inequality, we may} \me{write
\begin{align*}
&\mathbb{E}[\|\nabla f^\eta(\bar{\mathbf{x}}_k)\|^2|\mathcal{F}_k]]=\mathbb{E}[\| \nabla f^\eta(\bar{\mathbf{x}}_k)+\underline{\bar{\mathbf{y}}}_k-\underline{\bar{\mathbf{y}}}_k\|^2|\mathcal{F}_k]\\
&{\le} 2\mathbb{E}[\| \underline{\bar{\mathbf{y}}}_k-\nabla f^\eta(\bar{\mathbf{x}}_k)\|^2|\mathcal{F}_k]+2\mathbb{E}[\|\underline{\bar{\mathbf{y}}}_k\|^2|\mathcal{F}_k]\\
&{=} 2\mathbb{E}\left[\left\| \tfrac{1}{m} \textstyle \sum_{i=1}^m \nabla_\mathbf{x} f_i^\eta({x}_{i,k})-\tfrac{1}{m}\textstyle \sum_{i=1}^m \nabla_\mathbf{x} f_i^\eta(\bar{{x}}_k) \right\|^2|\mathcal{F}_k\right]\\
&+2\mathbb{E}[\|\underline{\bar{\mathbf{y}}}_k\|^2|\mathcal{F}_k]\\
 &\le  \tfrac{2}{m} \textstyle \sum_{i=1}^m\mathbb{E}\left[\left\| \nabla_\mathbf{x} f_i^\eta({x}_{i,k})- \nabla_\mathbf{x} f_i^\eta(\bar{{x}}_k)\right\|^2|\mathcal{F}_k\right]\\
 &+2\mathbb{E}[\|\underline{\bar{\mathbf{y}}}_k\|^2|\mathcal{F}_k]\\
&\overset{\scriptsize{\mbox{Lemma }} \ref{Properties of spherical smoothin}}{\le}
\tfrac{2L_0n}{m\eta}\mathbb{E}\left[\left\| \mathbf{x}_{k}-\mathbf{1}\bar{{\mathbf x}}_k\right\|^2|\mathcal{F}_k\right]+2\mathbb{E}[\|\underline{\bar{\mathbf{y}}}_k\|^2|\mathcal{F}_k].
\end{align*}
Therefore, \fy{we obtain}
\begin{align*}
\mathbb{E}[\| \underline{\bar{\mathbf{y}}}_k\|^2|\mathcal{F}_k]&\ge\tfrac{1}{2}\mathbb{E}[\|\nabla f^\eta(\bar{\mathbf{x}}_k)\|^2|\mathcal{F}_k]\\
&-\tfrac{L_0n}{m\eta}\mathbb{E}\left[\left\| \mathbf{x}_{k}-\mathbf{1}\bar{{\mathbf x}}_k\right\|^2|\mathcal{F}_k\right].
\end{align*}}
\end{proof} 
\begin{lemma}\label{Descent Lemma} Consider Algorithm \ref{alg:DZGT}. Let $\gamma \le \left(1-\tfrac{3 \beta}{2}\right)\tfrac{\eta}{2L_0n}$ where $\beta \in (0,\tfrac{2}{3})$. Under Assumptions~\ref{mixxx}--\ref{assum:main2}, for all $k\geq 0$, 
\begin{align*}
&\mathbb{E}[f^\eta(\bar{\mathbf{x}}_{k+1})]\le\mathbb{E}[ f^\eta(\bar{\mathbf{x}}_{k})]+\tfrac{2L_0n\gamma^2}{\eta}\mathbb{E}[\| \underline{{\mathbf{y}}}_k-\mathbf{1}\underline{\bar{\mathbf{y}}}_k\|^2]\\
&{+}\me{\left(\tfrac{\gamma}{2\beta}\tfrac{L_0^2n^2}{m\eta^2}-\tfrac{L_0n}{m\eta} \left(-\gamma + \tfrac{\gamma \beta}{2}+\tfrac{L_0n\gamma^2}{\eta}\right)\right)}\mathbb{E}\left[\left\|\mathbf{x}_k-\mathbf{1}\bar{\mathbf{x}_k}\right\|^2\right]\\
&+\tfrac{2L_0n\gamma^2}{\eta}  2m(n^2L_0^2+ \tfrac{4\tilde L_0^2n^2\varepsilon_0}{\eta^2})\left(1 + \tfrac{8(1+\me{\rho^2})}{ (1-\me{\rho^2})^2}\right)\\
&+\me{\left(\tfrac{-\gamma}{2} + \tfrac{3\gamma \beta}{4}+\tfrac{L_0n\gamma^2}{2\eta}\right)}\mathbb{E}\left[\|\nabla f^\eta(\bar{\mathbf{x}}_k)\|^2\right] +\tfrac{16\tilde L_0^2n^2\varepsilon_k}{\beta \eta^2}.
\end{align*}
\end{lemma}
\begin{proof}
From $\tfrac{L_0n}{2\eta}$-smoothness of $f^{\eta}$ in Lem.~\ref{Properties of spherical smoothin}, we have
\begin{align*}
&f^\eta(\bar{\mathbf{x}}_{k+1})\le f^\eta(\bar{\mathbf{x}}_{k})+\nabla f^\eta(\bar{\mathbf{x}}_k)^T(\bar{\mathbf{x}}_{k+1}-\bar{\mathbf{x}}_{k})\\
&+\tfrac{L_0n}{2\eta}\|\bar{\mathbf{x}}_{k+1}-\bar{\mathbf{x}}_{k}\|^2\\
&\overset{\footnotesize{\mbox{Lemma } }\ref{lemma:delta_x_in_terms_of_y}}{=} f^\eta(\bar{\mathbf{x}}_{k})-\gamma \nabla f^\eta(\bar{\mathbf{x}}_k)^T\underline{\bar{\mathbf{y}}}_k-\gamma \nabla f^\eta(\bar{\mathbf{x}}_k)^T\\
&\times \left(\tfrac{1}{m}\mathbf{1}^T(\mathbf{y}_k-\mathbf{1}\underline{\bar{\mathbf{y}}}_k)\right)+\tfrac{L_0n\gamma^2}{2\eta}\| \underline{\bar{\mathbf{y}}}_k+\tfrac{1}{m}\mathbf{1}^T(\mathbf{y}_k-\mathbf{1}\underline{\bar{\mathbf{y}}}_k)\|^2\\
&=f^\eta(\bar{\mathbf{x}}_{k})-\gamma \left(\nabla f^\eta(\bar{\mathbf{x}}_k)-\underline{\bar{\mathbf{y}}}_k+\underline{\bar{\mathbf{y}}}_k\right)^T\underline{\bar{\mathbf{y}}}_k\\
&-\gamma \nabla f^\eta(\bar{\mathbf{x}}_k)^T\left(\tfrac{1}{m}\mathbf{1}^T(\mathbf{y}_k-\mathbf{1}\underline{\bar{\mathbf{y}}}_k-\underline{{\mathbf{y}}}_k+\underline{{\mathbf{y}}}_k)\right)\\
&+\tfrac{L_0n\gamma^2}{2\eta}\| \underline{\bar{\mathbf{y}}}_k+\tfrac{1}{m}\mathbf{1}^T(\mathbf{y}_k-\mathbf{1}\underline{\bar{\mathbf{y}}}_k)\|^2.
\end{align*}
Invoking Lemma~\ref{lemma:delta_x_in_terms_of_y} (ii), for any $\beta>0$ we have
\begin{align*}
&f^\eta(\bar{\mathbf{x}}_{k+1})\le 
f^\eta(\bar{\mathbf{x}}_{k})-\gamma \left(\nabla f^\eta(\bar{\mathbf{x}}_k)-\underline{\bar{\mathbf{y}}}_k\right)^T\underline{\bar{\mathbf{y}}}_k-\gamma \|\underline{\bar{\mathbf{y}}}_k\|^2\\
&-\gamma \nabla f^\eta(\bar{\mathbf{x}}_k)^T\tfrac{1}{m}\mathbf{1}^T(\mathbf{y}_k-\underline{{\mathbf{y}}}_k)\\
&+\tfrac{L_0n\gamma^2}{2\eta}\| \underline{\bar{\mathbf{y}}}_k+ \tfrac{1}{m}\mathbf{1}^T(\mathbf{y}_k-\mathbf{1}\underline{\bar{\mathbf{y}}}_k)\|^2\\
&{\le}f^\eta(\bar{\mathbf{x}}_{k})+\tfrac{\gamma}{2\beta} \left\|\nabla f^\eta(\bar{\mathbf{x}}_k)-\underline{\bar{\mathbf{y}}}_k\right\|^2+\tfrac{\gamma \beta}{2}\|\underline{\bar{\mathbf{y}}}_k\|^2\\
&-\gamma \|\underline{\bar{\mathbf{y}}}_k\|^2-\gamma \nabla f^\eta(\bar{\mathbf{x}}_k)^T\tfrac{1}{m}\mathbf{1}^T(\mathbf{y}_k-\underline{{\mathbf{y}}}_k)\\
&+\tfrac{L_0n\gamma^2}{\eta}\| \underline{\bar{\mathbf{y}}}_k\|^2+\tfrac{L_0n\gamma^2}{\eta}\|  \tfrac{1}{m}\mathbf{1}^T(\mathbf{y}_k-\mathbf{1}\underline{\bar{\mathbf{y}}}_k)\|^2.
\end{align*}
Taking expectations on the both sides, we obtain 
\begin{align*}
&\mathbb{E}[f^\eta(\bar{\mathbf{x}}_{k+1})]  \le\mathbb{E}[ f^\eta(\bar{\mathbf{x}}_{k})]+\tfrac{\gamma}{2\beta} \mathbb{E}\left[\left\|\nabla f^\eta(\bar{\mathbf{x}}_k)-\underline{\bar{\mathbf{y}}}_k\right\|^2\right]\\
&+\tfrac{\gamma \beta}{2}\mathbb{E}\left[\|\underline{\bar{\mathbf{y}}}_k\|^2\right]-\gamma \mathbb{E}[\|\underline{\bar{\mathbf{y}}}_k\|^2]\\
&-\gamma \mathbb{E}\left[\nabla f^\eta(\bar{\mathbf{x}}_k)^T\tfrac{1}{m}\mathbf{1}^T(\mathbf{y}_k-\underline{{\mathbf{y}}}_k)\right]\\
&+\tfrac{L_0n\gamma^2}{\eta}\mathbb{E}[\| \underline{\bar{\mathbf{y}}}_k\|^2]+\tfrac{L_0n\gamma^2}{\eta}\mathbb{E}[\|  \mathbf{y}_k-\mathbf{1}\underline{\bar{\mathbf{y}}}_k \|^2].
\end{align*}
Using \eqref{eqn:compact_alg2}, \eqref{eqn:y_underbar}, \eqref{eqn:y_underbar1}, and Lemma~\ref{lemma:g_ik_eta_props}, we obtain
\begin{align*}
&\mathbb{E}[f^\eta(\bar{\mathbf{x}}_{k+1})]{\le}\mathbb{E}[ f^\eta(\bar{\mathbf{x}}_{k})]+\left(-\gamma + \tfrac{\gamma \beta}{2}+\tfrac{L_0n\gamma^2}{\eta}\right)\mathbb{E}[\| \underline{\bar{\mathbf{y}}}_k\|^2]\\
&+\tfrac{\gamma}{2\beta} \mathbb{E}\left[\left\|\nabla f^\eta(\bar{\mathbf{x}}_k)-\tfrac{1}{m} \textstyle\sum_{i=1}^m \nabla_\mathbf{x} f_i^\eta(\mathbf{x}_{i,k}) \right\|^2\right]\\
&+\tfrac{\gamma \beta}{2}\mathbb{E}\left[\|\nabla f^\eta(\bar{\mathbf{x}}_k)\|^2\right]+\tfrac{2}{\beta}\mathbb{E}\left[\|{\boldsymbol{\omega}}^{\eta,\varepsilon_{k+1}}_{k+1}\|^2\right]\\
&+\tfrac{2}{\beta}\mathbb{E}\left[\|{\boldsymbol{\omega}}^{\eta,\varepsilon_k}_k\|^2\right] +\tfrac{2L_0n\gamma^2}{\eta}\mathbb{E}[\| \mathbf{y}_k-\mathbf{1}\underline{\bar{\mathbf{y}}}_k\|^2+\|\underline{{\mathbf{y}}}_k-\underline{{\mathbf{y}}}_k\|^2]\\
&\overset{\scriptsize{\mbox{Lemma }}{\ref{lemm:omega_vareps}}}{\le}
\mathbb{E}[ f^\eta(\bar{\mathbf{x}}_{k})]+\left(-\gamma + \tfrac{\gamma \beta}{2}+\tfrac{L_0n\gamma^2}{\eta}\right)\mathbb{E}[\| \underline{\bar{\mathbf{y}}}_k\|^2]\\
&+\tfrac{\gamma}{2\beta}\tfrac{L_0^2n^2}{m\eta^2} \mathbb{E}\left[\left\|\mathbf{x}_k-\mathbf{1}\bar{\mathbf{x}_k}\right\|^2\right]+\tfrac{2L_0n\gamma^2}{\eta}\mathbb{E}[\| \underline{{\mathbf{y}}}_k-\mathbf{1}\underline{\bar{\mathbf{y}}}_k\|^2]\\
&+\tfrac{2L_0n\gamma^2}{\eta}\ 2m(n^2L_0^2+ \tfrac{4\tilde L_0^2n^2\varepsilon_0}{\eta^2})\left(1 + \tfrac{8(1+\me{\rho^2})}{ (1-\me{\rho^2})^2}\right)\\
&+\tfrac{\gamma \beta}{2}\mathbb{E}\left[\|\nabla f^\eta(\bar{\mathbf{x}}_k)\|^2\right]+\tfrac{4}{\beta}\left(\tfrac{4\tilde L_0^2n^2\varepsilon_k}{\eta^2}\right) .
\end{align*}
\me{Based on the bound for $\gamma$, we have $\left(-\gamma+\tfrac{\gamma \beta}{2}+\tfrac{L_0n\gamma^2}{\eta}\right)<0$. Invoking Lemma~\ref{lem:bound_bar_y_bar_squared}, we obtain
\begin{align*}
&\mathbb{E}[f^\eta(\bar{\mathbf{x}}_{k+1})]\le\mathbb{E}[ f^\eta(\bar{\mathbf{x}}_{k})]+\tfrac{2L_0n\gamma^2}{\eta}\mathbb{E}[\| \underline{{\mathbf{y}}}_k-\mathbf{1}\underline{\bar{\mathbf{y}}}_k\|^2]\\
&{+}\left(\tfrac{\gamma}{2\beta}\tfrac{L_0^2n^2}{m\eta^2}-\tfrac{L_0n}{m\eta} \left(-\gamma + \tfrac{\gamma \beta}{2}+\tfrac{L_0n\gamma^2}{\eta}\right)\right) \mathbb{E}\left[\left\|\mathbf{x}_k-\mathbf{1}\bar{\mathbf{x}_k}\right\|^2\right]\\
&+\tfrac{2L_0n\gamma^2}{\eta}  2m(n^2L_0^2+ \tfrac{4\tilde L_0^2n^2\varepsilon_0}{\eta^2})\left(1 + \tfrac{8(1+\me{\rho^2})}{ (1-\me{\rho^2})^2}\right)\\
&+\left(\tfrac{-\gamma}{2} + \tfrac{3\gamma \beta}{4}+\tfrac{L_0n\gamma^2}{2\eta}\right)\mathbb{E}\left[\|\nabla f^\eta(\bar{\mathbf{x}}_k)\|^2\right] +\tfrac{16\tilde L_0^2n^2\varepsilon_k}{\beta \eta^2} .
\end{align*}}
\end{proof}
\begin{lemma}\label{Iterates Contraction} Consider Algorithm~\ref{alg:DZGT}. Let $\beta >0$. Under Assumptions~\ref{mixxx}--\ref{assum:main2}, for all $k\geq 0$, the following holds.
\begin{align*}
&\mbox{(i) } \mathbb{E}[\|\mathbf{x}_{k+1}-\mathbf{1}\bar {\mathbf{x}}_{k+1}\|^2]\le(1+\beta)\rho^2 \mathbb{E}\left[\|\mathbf{x}_{k}-\mathbf{1}\bar {\mathbf{x}}_{k}\|^2\right]\\
&+3(1+\tfrac{1}{\beta})\gamma^2 \mathbb{E}[\|\underline{\mathbf{y}}_{k}-\mathbf{1}\underline{\bar {\mathbf{y}}}_{k}\|^2]\\
&+6(1+\tfrac{1}{\beta})\gamma^2 2m(n^2L_0^2+ \tfrac{4\tilde L_0^2n^2\varepsilon_0}{\eta^2})\left(1 + \tfrac{8(1+\me{\rho^2})}{ (1-\me{\rho^2})^2}\right).\\
&\mbox{(ii) }\mathbb{E}[\|\underline{\mathbf{y}}_{k+1}-\mathbf{1}\underline{\bar {\mathbf{y}}}_{k+1}\|^2]\le  \left((1+\beta)\rho^2+\tfrac{4L_0^2n^2\gamma^2}{\eta^2}(1+\tfrac{1}{\beta})^2\right)\\
\begin{split}
&\mathbb{E}[\|\underline{ \mathbf{y}}_k-\mathbf{1}\underline{\bar{ \mathbf{y}}}_k\|^2]+\Big(\tfrac{4L_0^2n^2}{\eta^2}(1+\tfrac{1}{\beta})^2\\
&+\tfrac{L_0^2n^2}{\eta^2}\rho^2(1+\beta)(1+\tfrac{1}{\beta})+\tfrac{2L_0n}{m\eta}\tfrac{4L_0^2n^2\gamma^2}{\eta^2}(1+\tfrac{1}{\beta})^2\Big)
\end{split}\\
&\mathbb{E}\left[\|\mathbf{x}_k-\mathbf{1}\bar{\mathbf{x}}_k\|^2\right]\\
&+\tfrac{4L_0^2n^2\gamma^2}{\eta^2}(1+\tfrac{1}{\beta})^2 2m(n^2L_0^2+ \tfrac{4\tilde L_0^2n^2\varepsilon_0}{\eta^2})\left(1 + \tfrac{8(1+\me{\rho^2})}{ (1-\me{\rho^2})^2}\right)\\
&+\tfrac{8L_0^2n^2\gamma^2}{\eta^2}(1+\tfrac{1}{\beta})^2\mathbb{E}\left[\|\nabla f^\eta(\bar{\mathbf{x}}_k)\|^2\right].
\end{align*}
\begin{proof}
(i) From Assumption~\ref{mixxx}, for $\beta>0$ we have
\begin{align*}
&\|\mathbf{x}_{k+1}-\mathbf{1}\bar {\mathbf{x}}_{k+1}\|^2=\|\mathbf{w}\mathbf{x}_k-\gamma \mathbf{y}_k-\mathbf{1}(\bar{\mathbf{x}}_k - {\gamma}\bar{\mathbf{y}}_k)\|^2\\
&\le (1+\beta)\|\mathbf{w}\mathbf{x}_k-\mathbf{1}\bar{\mathbf{x}}_k\|^2+ (1+\tfrac{1}{\beta})\gamma^2\|\mathbf{y}_k-\mathbf{1}\bar{\mathbf{y}}_k\|^2\\
&\le  (1+\beta)\rho^2\|\mathbf{x}_k-\mathbf{1}\bar{\mathbf{x}}_k\|^2+ 3(1+\tfrac{1}{\beta})\gamma^2\|\mathbf{y}_k-\underline{\mathbf{y}}_k\|^2\\
&+3(1+\tfrac{1}{\beta})\gamma^2\|\underline{\mathbf{y}}_k-\mathbf{1}\underline{\bar{\mathbf{y}}}_k\|^2+3(1+\tfrac{1}{\beta})\gamma^2\|\mathbf{1}\bar{\mathbf{y}}_k-\mathbf{1}\underline{\bar{\mathbf{y}}}_k\|^2.
\end{align*}
Taking expectations on both sides and using Lemma~\ref{lemm:omega_vareps}, we obtain the inequality in (i).

\noindent (ii) From Assumption~\ref{mixxx} and equation~\eqref{eqn:y_underbar}, we have
\begin{align*}
\begin{split}
&\|\underline{\mathbf{y}}_{k+1}-\mathbf{1}\underline{\bar {\mathbf{y}}}_{k+1}\|^2\le \Big\|\mathbf{w}\underline{ \mathbf{y}}_k +\nabla_{\mathbf{x}} \mathbf{f}^\eta(\mathbf{x}_{k+1})-\nabla_{\mathbf{x}} \mathbf{f}^\eta(\mathbf{x}_{k})\\
&-\tfrac{1}{m}\mathbf{1}\mathbf{1}^T\left(\mathbf{w}\underline{ \mathbf{y}}_k +\nabla_{\mathbf{x}} \mathbf{f}^\eta(\mathbf{x}_{k+1})-\nabla_{\mathbf{x}} \mathbf{f}^\eta(\mathbf{x}_{k})\right)\Big\|^2
\end{split}\\
&\le (1+\beta)\rho^2\|\underline{ \mathbf{y}}_k-\mathbf{1}\underline{\bar{ \mathbf{y}}}_k\|^2\\
&+(1+\tfrac{1}{\beta})\|\nabla_{\mathbf{x}} \mathbf{f}^\eta(\mathbf{x}_{k+1})-\nabla_{\mathbf{x}} \mathbf{f}^\eta(\mathbf{x}_{k})\|^2
\end{align*} 
We bound the second term as follows. 
\begin{align*}
&\|\nabla_{\mathbf{x}} \mathbf{f}^\eta(\mathbf{x}_{k+1})-\nabla_{\mathbf{x}} \mathbf{f}^\eta(\mathbf{x}_{k})\|^2 \le \tfrac{L_0^2n^2}{\eta^2 }\|\mathbf{x}_{k+1}-\mathbf{x}_{k}\|^2\\
&=\tfrac{L_0^2n^2}{\eta^2 }\| \mathbf{w}(\mathbf{x}_k-\mathbf{1}\bar{\mathbf{x}}_k)+\mathbf{1}\bar{\mathbf{x}}_k-\mathbf{x}_{k}-\gamma \mathbf{y}_k\|^2\\
&\le \tfrac{L_0^2n^2}{\eta^2 }\rho^2(1+\beta)\|\mathbf{x}_k-\mathbf{1}\bar{\mathbf{x}}_k\|^2\\
&+\tfrac{L_0^2n^2}{\eta^2 }(1+\tfrac{1}{\beta})\|\mathbf{1}\bar{\mathbf{x}}_k-\mathbf{x}_{k}-\gamma \mathbf{y}_k\|^2\\
&\le  \tfrac{L_0^2n^2}{\eta^2 }\rho^2(1+\beta)\|\mathbf{x}_k-\mathbf{1}\bar{\mathbf{x}}_k\|^2+\tfrac{4L_0^2n^2(1+\tfrac{1}{\beta})}{\eta^2 }\|\mathbf{x}_k-\mathbf{1}\bar{\mathbf{x}}_k\|^2\\
&+\tfrac{4L_0^2n^2\gamma^2}{\eta^2 }(1+\tfrac{1}{\beta})\|\mathbf{y}_k-\underline{\mathbf{y}}_k\|^2+\tfrac{4L_0^2n^2\gamma^2(1+\tfrac{1}{\beta})}{\eta^2 }\|\underline{\mathbf{y}}_k-\mathbf{1}\underline{\bar{\mathbf{y}}}_k\|^2\\
&+\tfrac{4L_0^2n^2\gamma^2}{\eta^2 }(1+\tfrac{1}{\beta})\|\mathbf{1}\underline{\bar{\mathbf{y}}}_k\|^2.
\end{align*}
From the two preceding inequalities, we obtain
\begin{align*}
&\|\underline{\mathbf{y}}_{k+1}-\mathbf{1}\underline{\bar {\mathbf{y}}}_{k+1}\|^2\\
&\le \left((1+\beta)\rho^2+\tfrac{4L_0^2n^2\gamma^2}{\eta^2}(1+\tfrac{1}{\beta})^2\right)\|\underline{ \mathbf{y}}_k-\mathbf{1}\underline{\bar{ \mathbf{y}}}_k\|^2\\
&+\left(\tfrac{4L_0^2n^2}{\eta^2 }(1+\tfrac{1}{\beta})^2+\tfrac{L_0^2n^2}{\eta^2 }\rho^2(1+\beta)(1+\tfrac{1}{\beta})\right)\|\mathbf{x}_k-\mathbf{1}\bar{\mathbf{x}}_k\|^2\\
&+\tfrac{4L_0^2n^2\gamma^2}{\eta^2 }(1+\tfrac{1}{\beta})^2\|\mathbf{y}_k-\underline{\mathbf{y}}_k\|^2+\tfrac{4L_0^2n^2\gamma^2}{\eta^2 }(1+\tfrac{1}{\beta})^2\|\underline{\bar{\mathbf{y}}}_k\|^2.
\end{align*}
Taking expectations on both sides. We have
\begin{align*}
&\mathbb{E}\left[\|\underline{\mathbf{y}}_{k+1}-\mathbf{1}\underline{\bar {\mathbf{y}}}_{k+1}\|^2\right]\le \left((1+\beta)\rho^2+\tfrac{4L_0^2n^2\gamma^2}{\eta^2 }(1+\tfrac{1}{\beta})^2\right)\\
\begin{split}
&\mathbb{E}\left[\|\underline{ \mathbf{y}}_k-\mathbf{1}\underline{\bar{ \mathbf{y}}}_k\|^2\right]+\Big(\tfrac{4L_0^2n^2}{\eta^2 }(1+\tfrac{1}{\beta})^2\\
&+\tfrac{L_0^2n^2}{\eta^2 }\rho^2(1+\beta)(1+\tfrac{1}{\beta})\Big)\mathbb{E}\left[\|\mathbf{x}_k-\mathbf{1}\bar{\mathbf{x}}_k\|^2\right]+\tfrac{4L_0^2n^2\gamma^2}{\eta^2}
\end{split}\\
&(1+\tfrac{1}{\beta})^2\mathbb{E}\left[\|\mathbf{y}_k-\underline{\mathbf{y}}_k\|^2\right]+\tfrac{4L_0^2n^2\gamma^2}{\eta^2}(1+\tfrac{1}{\beta})^2\mathbb{E}\left[\|\underline{\bar{\mathbf{y}}}_k\|^2\right]\\
&\overset{\scriptsize{\mbox{Lemma }}{\ref{lemm:omega_vareps}}}{\le} \left((1+\beta)\rho^2+\tfrac{4L_0^2n^2\gamma^2}{\eta^2}(1+\tfrac{1}{\beta})^2\right)\mathbb{E}\left[\|\underline{ \mathbf{y}}_k-\mathbf{1}\underline{\bar{ \mathbf{y}}}_k\|^2\right]\\
&+\left(\tfrac{4L_0^2n^2}{\eta^2}(1+\tfrac{1}{\beta})^2+\tfrac{L_0^2n^2}{\eta^2}\rho^2(1+\beta)(1+\tfrac{1}{\beta})\right)\\
&\mathbb{E}\left[\|\mathbf{x}_k-\mathbf{1}\bar{\mathbf{x}}_k\|^2\right]+\tfrac{4L_0^2n^2\gamma^2}{\eta^2}(1+\tfrac{1}{\beta})^2 2m(n^2L_0^2+ \tfrac{4\tilde L_0^2n^2\varepsilon_0}{\eta^2})\\
&\left(1 + \tfrac{8(1+\me{\rho^2})}{ (1-\me{\rho^2})^2}\right)+\tfrac{4L_0^2n^2\gamma^2}{\eta^2}(1+\tfrac{1}{\beta})^2\mathbb{E}\left[\|\underline{\bar{\mathbf{y}}}_k\|^2\right].
\end{align*}
Invoking Lemma~\ref{main-parameter}, we obtain the inequality in (ii). 
\end{proof}
\end{lemma}
In the following, we introduce a Lyapunov function~\cite{lu2019gnsd} that helps with obtaining the rate statements. The proof follows from Lemma~\ref{Descent Lemma} and Lemma~\ref{Iterates Contraction} and is omitted.
\begin{lemma}[Lyapunov \fy{f}unction] Consider the following function, for some $Q >0$, and $k\geq 0$.
\begin{align*}
\mathbf{L}(\mathbf{x}_k)\triangleq \mathbb{E}[f^\eta(\bar{\mathbf{x}}_k)]+\mathbb{E}[\|\mathbf{x}_k-\mathbf{1}\bar{\mathbf{x}}_k\|^2]+Q\mathbb{E}[\|\underline{ \mathbf{y}}_k-\mathbf{1}\underline{\bar{ \mathbf{y}}}_k\|^2].
\end{align*}
Then, for any $k\geq 0$ we have
\begin{align}
&\mathbf{L}(\mathbf{x}_{k+1})-\mathbf{L}(\mathbf{x}_k) \le -C_1\gamma\mathbb{E}\left[\|\nabla f^\eta(\bar{\mathbf{x}}_k)\|^2\right]\label{Lyapunov Function}\\
&-C_2\mathbb{E}\left[\|\mathbf{x}_k-\mathbf{1}\bar{\mathbf{x}}_k\|^2\right]-C_3\mathbb{E}[\|\underline{ \mathbf{y}}_k-\mathbf{1}\underline{\bar{ \mathbf{y}}}_k\|^2]+C_4\gamma^2+C_{5,k}\notag,
\end{align}
where the scalars $C_1,\ldots,C_4$ and $C_{5,k}$ are defined as follows. 
\small\begin{align*}
&C_1\triangleq \me{\tfrac{1}{2}-\tfrac{3\beta}{4}-\tfrac{L_0n\gamma}{2\eta}}-\tfrac{8QL_0^2n^2\gamma}{\eta^2}\left(1+\tfrac{1}{\beta}\right)^2,\\
&C_2 \triangleq 1-\tfrac{\gamma}{2\beta}\tfrac{L_0^2n^2}{m\eta^2}\me{+\tfrac{L_0n}{m\eta} \left(-\gamma + \tfrac{\gamma \beta}{2}+\tfrac{L_0n\gamma^2}{\eta}\right)}-(1+\beta)\rho^2 \\
&-\tfrac{4QL_0^2n^2}{\eta^2}(1+\tfrac{1}{\beta})^2-\tfrac{QL_0^2n^2}{\eta^2}\rho^2(1+\beta)(1+\tfrac{1}{\beta})\\
&\fy{-\tfrac{8QL_0^3n^3\gamma^2}{m\eta^3 }(1+\tfrac{1}{\beta})^2},\\
&C_3 \triangleq Q-\tfrac{2L_0n\gamma^2}{\eta}-3(1+\tfrac{1}{\beta})\gamma^2-Q(1+\beta)\rho^2\\
&-\tfrac{4QL_0^2n^2\gamma^2}{\eta^2}(1+\tfrac{1}{\beta})^2,\\
&C_4 \triangleq \left(\tfrac{2L_0n\gamma^2}{\eta}+6(1+\tfrac{1}{\beta})\gamma^2+\tfrac{4L_0^2n^2\gamma^2}{\eta^2}(1+\tfrac{1}{\beta})^2\right)\\
&\times 2m(n^2L_0^2+ \tfrac{4\tilde L_0^2n^2\varepsilon_0}{\eta^2})\left(1 + \tfrac{8(1+\me{\rho^2})}{ (1-\me{\rho^2})^2}\right),\\
&C_{5,k} \triangleq \theta \varepsilon_k,\text{\fy{where} } \theta \triangleq \tfrac{4}{\beta}\left(\tfrac{4 \tilde{L}_0^2n^2}{\eta^2}\right).
\end{align*}
\end{lemma}
\me{
\begin{proof}
Using the definition of the Lyapunov \fy{function,} the bound for $\mathbb{E}[f^\eta(\bar{\mathbf{x}}_{k+1})]$ in Lemma~\ref{Descent Lemma}, \fy{and the} bounds for $ \mathbb{E}[\|\mathbf{x}_{k+1}-\mathbf{1}\bar {\mathbf{x}}_{k+1}\|^2]$ and $\mathbb{E}[\|\underline{\mathbf{y}}_{k+1}-\mathbf{1}\underline{\bar {\mathbf{y}}}_{k+1}\|^2]$ in Lemma~\ref{Iterates Contraction}, we can \fy{obtain the result}.
\end{proof}}
The main convergence rate statement is presented as follows, where we show that Algorithm~\ref{alg:DZGT}-\ref{alg:lowerlevel} admits an iteration complexity of $\mathcal{O}(\epsilon^{-2})$ for both the mean-square of consensus \fy{error} metric and the mean-square of an aggregate gradient of the smoothed implicit function.
\begin{theorem}\label{thm:main} Consider Algorithm \ref{alg:DZGT}. Let $\gamma:=\tfrac{C_0}{\sqrt{K}}$ where $C_0 \triangleq \min\{T_1,T_2,T_3\}$. Let Assumptions~\ref{mixxx}--\ref{assum:main2} hold. Suppose $\beta \me{\ \in\ } (0,\min\{{\tfrac{2}{3}},\rho^{-2}-1 \})$, $Q:=\alpha \gamma$ for \fy{some $\alpha>0$ such that $\alpha > 0.25(1+\frac{1}{\beta})^{-2}(1-\tfrac{3 \beta}{2})^{-1}$}, and  
\begin{align*}
&T_1\triangleq \tfrac{\sqrt{1+\me{32}\alpha(\me{2}-\me3\beta)(1+\tfrac{1}{\beta})^2}-1}{\tfrac{\me{16}L_0n\alpha}{\eta}(1+\tfrac{1}{\beta})\me{^2}}, \quad  \me{T_2\triangleq \tfrac{\fy{-b-\sqrt{b^2-4ac}}}{2a}},\\
&T_3 \triangleq \tfrac{\fy{\bar{b}+}\sqrt{\fy{\bar{b}}^2+\frac{16\alpha L_0^2n^2}{\eta^2}(1+\frac{1}{\beta})^2\fy{(\alpha c)}  }}{\frac{8L_0^2n^2\alpha}{\eta^2}(1+\frac{1}{\beta})^2},
\end{align*}
where we \fy{define scalars $a,b, c, \bar{b}$ as }
\begin{align*}
&a\triangleq \fy{-\tfrac{L_0^2n^2 }{m \eta^2}\left( 4\alpha(1+\tfrac{1}{\beta})^2  (1-\tfrac{3 \beta}{2} )-1 \right)},\\
&b\triangleq-\tfrac{L_0^2n^2}{2\beta m\eta^2}\fy{-}\tfrac{L_0n  }{m\eta} \fy{ (1 - \tfrac{  \beta}{2} )}-\tfrac{4\alpha L_0^2n^2}{\eta^2}(1+\tfrac{1}{\beta})^2\\
&-\tfrac{\alpha L_0^2n^2}{\eta^2}\rho^2(1+\beta)(1+\tfrac{1}{\beta}),\\
&c\triangleq  1-(1+\beta)\rho^2 , \quad \fy{\bar{b}\triangleq -}\left(\tfrac{2L_0n}{\eta}+3(1+\tfrac{1}{\beta})\right).
\end{align*}
Then, the following holds for $K \geq C_0^2\left(1-\tfrac{3 \beta}{2}\right)^{-2}\tfrac{4L_0^2n^2}{\eta^2}$.
\begin{align}
&C_1\me{\mathbb{E}\left[\|\nabla f^\eta(\bar{\mathbf{x}}_k)\|^2\right]}+\tfrac{C_2\fy{\sqrt{K}}}{C_0}\mathbb{E}\left[\|\mathbf{x}_k-\mathbf{1}\bar{\mathbf{x}}_k\|^2\right]\notag \\
&\fy{\le \ }\me{\left(\tfrac{\mathbf{L}_0-\underline{\mathbf{L}}}{C_0}+C_4C_0+2\sqrt{2}\theta \mathcal{O}(1)\right)\tfrac{1}{\sqrt{K}}}\label{ine::bound},
\end{align}
\me{where $\underline{\mathbf{L}} \fy{ \ \triangleq -L_0\eta +\inf_{x} \, f(x)}$ and $\mathbf{L}_0\triangleq \mathbf{L}(\fy{{\bf x}_0})$.} 
\end{theorem}
\begin{proof} \me{First, we \fy{show that} $C_1, C_2$, and $C_3$ \fy{are non-negative. Recall that} $Q:=\alpha \gamma$ \fy{implying that} $C_1$ is a quadratic expression in terms of $\gamma$ and the coefficient of $\gamma^2$ is negative. \fy{This implies that the term $C_0$ is} positive between \fy{the two} roots. \fy{Note that} one of the roots is negative, \fy{while the other root is equal to} $T_1$. \fy{This implies that for $0<  \gamma \le T_1$,  we have $C_1\ge 0$}}.  
 Next, we show that $C_2 \geq 0$. From the bound on $K$ and the choice of $\gamma$, we have $\gamma \le \left(1-\tfrac{3 \beta}{2}\right)\tfrac{\eta}{2L_0n}$. This implies that
\begin{align}\label{eqn:boundingC2}
& \fy{-\tfrac{8QL_0^3n^3\gamma^2}{m\eta^3 }(1+\tfrac{1}{\beta})^2  \geq -\tfrac{4QL_0^2n^2\gamma}{m \eta^2 }(1+\tfrac{1}{\beta})^2  (1-\tfrac{3 \beta}{2} )} .
\end{align}
\me{\fy{Let us} define a new term $\fy{\hat{C}_2}$ as  
\begin{align*}
\fy{\hat{C}_2} &\triangleq 1-(1+\beta)\rho^2-\tfrac{\gamma}{2\beta}\tfrac{L_0^2n^2}{m\eta^2}\fy{-}\tfrac{L_0n \fy{\gamma}}{m\eta} \fy{ (1 - \tfrac{  \beta}{2} )}\\
& -\tfrac{4\fy{\alpha \gamma}L_0^2n^2}{\eta^2}(1+\tfrac{1}{\beta})^2-\tfrac{\fy{\alpha \gamma}L_0^2n^2}{\eta^2}\rho^2(1+\beta)(1+\tfrac{1}{\beta})\\
&\fy{+\tfrac{L_0^2n^2\gamma^2}{m \eta^2} -\tfrac{4\fy{\alpha}L_0^2n^2\gamma^2}{m \eta^2 }(1+\tfrac{1}{\beta})^2  (1-\tfrac{3 \beta}{2} )}.
\end{align*}}
\noindent \fy{Invoking \eqref{eqn:boundingC2} and $Q=\alpha \gamma$, we have} $C_2\ge \fy{\hat{C}_2}$. \fy{Note that $\hat{C}_2$} admits a quadratic expression in terms of $\gamma$. \fy{Also, the assumptions on $\beta$ and $\alpha$ imply that $1-(1+\beta)\rho^2 >0$ and $1-4\alpha(1+\frac{1}{\beta})^2 (1-\frac{3 \beta}{2} )<0$. Thus, we have $a<0$, $b<0$, and $c>0$. From the definition of $T_2$ and invoking $0< \gamma \le T_2$, we have} that $C_2 \geq 0$. \fy{To show that $C_3\geq 0$, we can write
\begin{align*}
 \hat{C}_3 \triangleq \tfrac{C_3}{\gamma}&=\alpha (1-(1+\beta)\rho^2)-\left(\tfrac{2L_0n }{\eta}+3(1+\tfrac{1}{\beta})\right)\gamma \\
&- \tfrac{4\alpha L_0^2n^2 }{\eta^2}(1+\tfrac{1}{\beta})^2\gamma^2.
\end{align*}} 
The non-negativity of \fy{$\hat{C}_3$ can be shown by} invoking $0\fy{ \ <\ } \gamma \le T_3$. \fy{Next, consider \eqref{Lyapunov Function}.} Summing both sides of $\eqref{Lyapunov Function}$ over $ k=0,1,\ldots,K-1$, where $K\ge 1$, we obtain
\begin{align*}
&\underline{\mathbf{L}}-\mathbf{L}_0 \fy{\ \le \mathbf{L}(\mathbf{x}_K) -\mathbf{L}_0} \le -C_1K\gamma\me{\mathbb{E}\left[\|\nabla f^\eta(\bar{\mathbf{x}}_k)\|^2\right]}\\
&-C_2K\mathbb{E}\left[\|\mathbf{x}_k-\mathbf{1}\bar{\mathbf{x}}_k\|^2\right]\\
&-C_3K\mathbb{E}[\|\underline{ \mathbf{y}}_k-\mathbf{1}\underline{\bar{ \mathbf{y}}}_k\|^2]+C_4K\gamma^2+\textstyle \sum_{k=0}^{K-1}C_{5,k}.
\end{align*}
\fy{Rearranging the terms}, we obtain
\begin{align*}
&C_1\gamma\mathbb{E}[\|\nabla f^\eta(\bar{\mathbf{x}}_k)\|^2]+C_2\mathbb{E}\fy{[}\|\mathbf{x}_k-\mathbf{1}\bar{\mathbf{x}}_k\|^2\fy{]}  \\
&\le \tfrac{\underline{\mathbf{L}}-\mathbf{L}_0}{K}+C_4\gamma^2+\tfrac{\textstyle \sum_{k=0}^{K-1}C_{5,k}}{K}.
\end{align*}
\fy{Dividing both sides by $\gamma$ and substituting $\gamma :=\tfrac{C_0}{\sqrt{K}}$, we have}
\begin{align*}
&C_1\mathbb{E}[\|\nabla f^\eta(\bar{\mathbf{x}}_k)\|^2]+\tfrac{C_2\fy{\sqrt{K}}}{C_0}\mathbb{E}\fy{[}\|\mathbf{x}_k-\mathbf{1}\bar{\mathbf{x}}_k\|^2\fy{]}  \\
&\le  \left(\tfrac{\underline{\mathbf{L}}-\mathbf{L}_0}{C_0}+C_4C_0\right)\tfrac{1}{\sqrt{K}}+\tfrac{\textstyle \sum_{k=0}^{K-1}\theta \varepsilon_k}{K}.
\end{align*}
\me{Consider Algorithm~\ref{alg:lowerlevel}. From \cite[Theorem 2]{cui2104complexity}, and that $t_k:= \sqrt{k+1}$, we have that $\varepsilon_k = \tfrac{\mathcal{O}(1)}{\sqrt{k+1}+\Gamma}$. \fy{We} obtain
\me{\begin{align*}
&C_1\mathbb{E}[\|\nabla f^\eta(\bar{\mathbf{x}}_k)\|^2]+\tfrac{C_2\fy{\sqrt{K}}}{C_0}\mathbb{E}[\|\mathbf{x}_k-\mathbf{1}\bar{\mathbf{x}}_k\|^2]  \\
&\le  \left(\tfrac{\underline{\mathbf{L}}-\mathbf{L}_0}{C_0}+C_4C_0\right)\tfrac{1}{\sqrt{K}}+\tfrac{\textstyle \sum_{k=0}^{K-1} \tfrac{\theta \mathcal{O}(1)}{\sqrt{k+1}+\Gamma}}{K}.
\end{align*}}
\me{\fy{From}~\cite[Lemma 9 (b)]{yousefian2017smoothing}, we have
\begin{align*}
\fy{\textstyle\sum_{k=0}^{K-1}}\tfrac{1}{\sqrt{k+1}}\le{2\sqrt{K+1}-1}.
\end{align*}}
\me{Invoking this bound, we obtain the result.}}
\end{proof}
\me{\begin{remark}
The infeasibility \fy{of the equilibrium constraints in \eqref{eqn:prob} incurred by our method at iteration $k$,} is the difference between the exact solution \fy{to} the VI problem, \fy{denoted by $z(\bullet)$, and the inexact} solution \fy{computed by Alg.~\ref{alg:lowerlevel}, denoted by $z_{\varepsilon_k}(\bullet)$}. \fy{This is indeed quantified by} $\varepsilon_k=\tfrac{\mathcal{O}(1)}{\sqrt{k+1}+\Gamma}$ in \fy{view of} Lemma~\ref{lemm:omega_vareps} \fy{and the discussion in the proof of Thm.~\ref{thm:main}}. 
\end{remark}}

\section{Numerical Results}\label{sec:num}
In this section, we present preliminary experiments to validate the theoretical convergence of the proposed scheme. We compare the performance of the algorithm with that of the ZSOL-ncvx method (Algorithm 3 in \cite{cui2104complexity}). ZSOL-ncvx is a \fy{zeroth-order} method that can be viewed as a centralized counterpart of our scheme. We consider a bilevel optimization problem with the form
\begin{align*}
\min_x \ \tfrac{1}{m}\textstyle \sum_{i=1}^m\mathbb{E}[-x_1^2-3x_2-\xi(\omega)y_1(x)+\left(y_2(x)\right)^2],
\end{align*}
where $y(x)$ is the unique solution to the following parametric optimization problem.
\begin{align*}
&\min_y \  \mathbb{E}[\, 2x_1^2+y_1^2+y_2^2-\zeta(\omega)y_2\, ]\\
 &  \hbox{s.t.} \quad     x_1^2-2x_1+x_2^2-2y_1+y_2\ge -3, \notag \\
  & \ \ \ \ \quad x_2+3y_1-y_2\ge 4,\notag \\
  &\ \ \ \ \quad y_1,y_2\ge 0\notag .
\end{align*}
Notably, the constraints of the lower-level problem \fy{are characterized by} the upper-level decisions $x$. 

\textbf{Problem and algorithm parameters.} We assume that both \fy{$\xi$ and $\zeta $} \fy{are normally distributed}. We run rs-DZGT for $100$ iterations of the upper-level scheme, e.g., Alg.~\ref{alg:DZGT}, and use $\gamma \in \{10^{-5}, 10^{-6}\} $. In addition, Alg.~\ref{alg:lowerlevel} is terminated after $\sqrt{k+1}$ iterations where $k$ denotes the iteration index of Alg.~\ref{alg:DZGT}. Furthermore, for the network, we choose three settings for the mixing matrix $\bold{w}$: ring graph, \fy{a} sparse graph, and complete graph.

\textbf{Evaluation of the implicit objective function.} For each method and setting, we run the scheme five times and report the sample mean of the global objective function. Notably, to evaluate the objective function at each epoch, we use an approximation of $y(x)$ by running the projected stochastic gradient method, i.e., Alg.~\ref{alg:lowerlevel}. 

\textbf{Insights.} The implementation results are presented in Figure~\ref{fig:comparison}, \me{where the x-axis \fy{denotes} the number of iterations in the upper-level problem}. We observe that rs-DZGT appears to be more robust to the choice of the network. When the network size increases, the performance of our method does not degrade significantly. Also, by increasing the connectivity of the network among the agents, our method performs better. This is more clear from the consensus \fy{error} data provided in Table~\ref{table_example} and Table~\ref{table_example2}. Lastly, we note that rs-DZGT \fy{displays a small sensitivity} with respect to the \fy{two choices} of $\gamma$ and performs relatively close to its centralized counterpart in almost all cases. 
\me{\begin{remark}
\fy{The higher the network} connectivity, \fy{the smaller the parameter} $\rho$~\cite{pu2021distributed}, \fy{that explicitly appears in the terms} $C_1$, $C_2$, and $C_4$. \fy{Note that the smaller $\rho$, the larger the error bound in Thm.~\ref{thm:main}.}
\end{remark}}
\begin{figure*}
\centering{
\begin{table}[H]
\setlength{\tabcolsep}{0pt}
\centering
 \begin{tabular}{c || c  c  c }
  {\footnotesize {Setting}\ \ }& {\footnotesize  Ring graph} & {\footnotesize Sparse graph} & {\footnotesize Complete graph  } \\
 \hline\\
\rotatebox[origin=c]{90}{{\footnotesize {Sample \fy{ave.} implicit objective}}}
&
\begin{minipage}{.30\textwidth}
\includegraphics[scale=.199, angle=0]{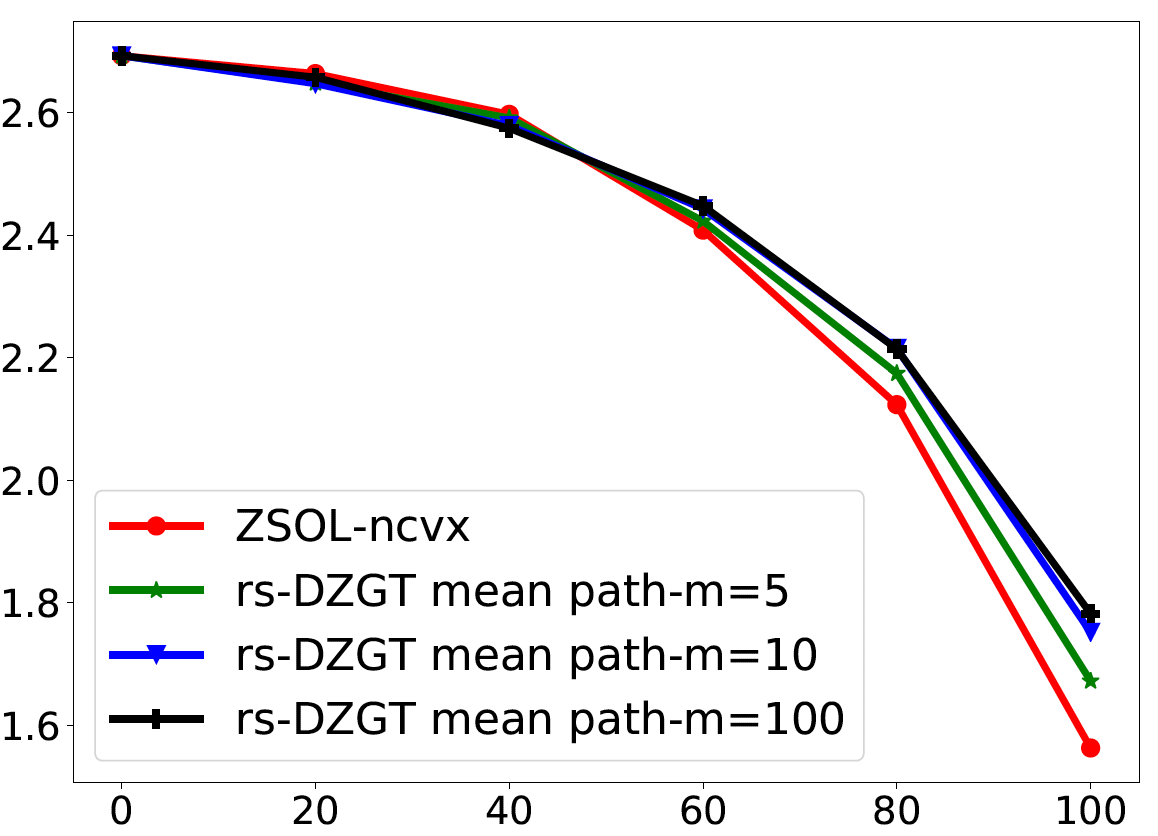}
\end{minipage}
&
\begin{minipage}{.30\textwidth}
\includegraphics[scale=.199, angle=0]{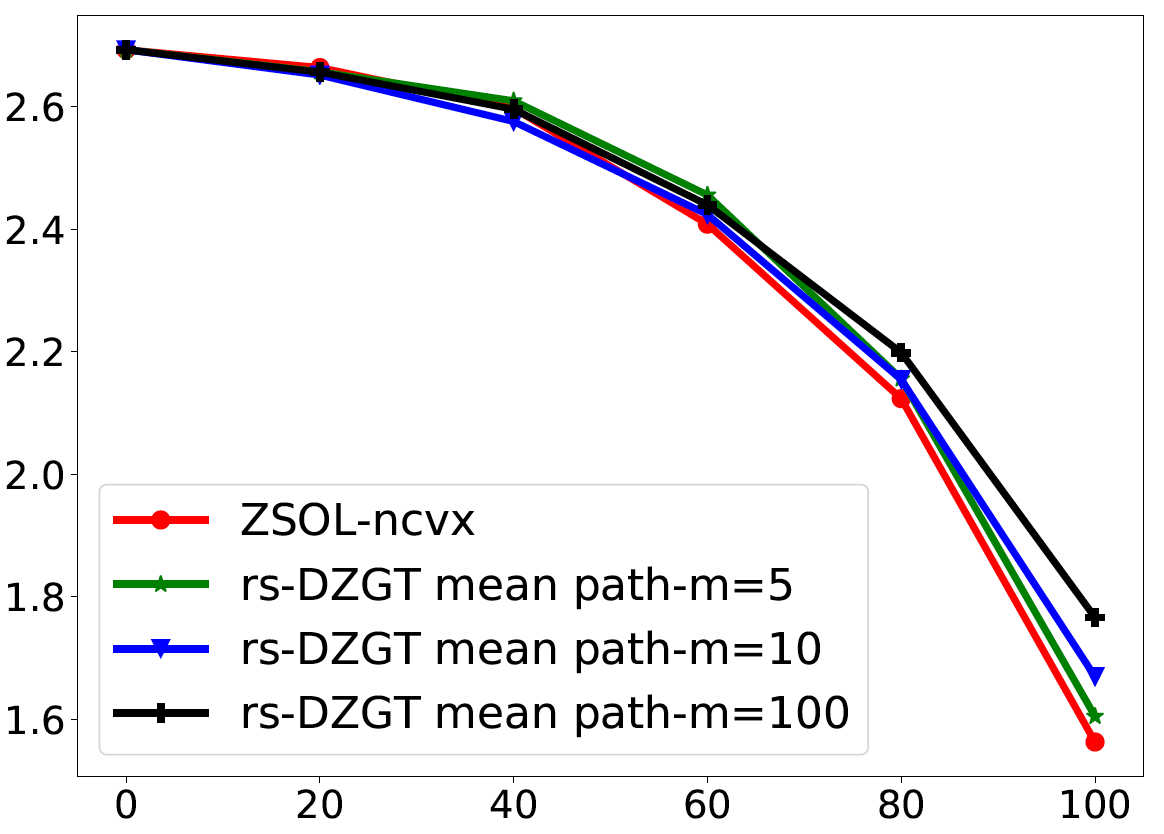}
\end{minipage}
	&
\begin{minipage}{.30\textwidth}
\includegraphics[scale=.199, angle=0]{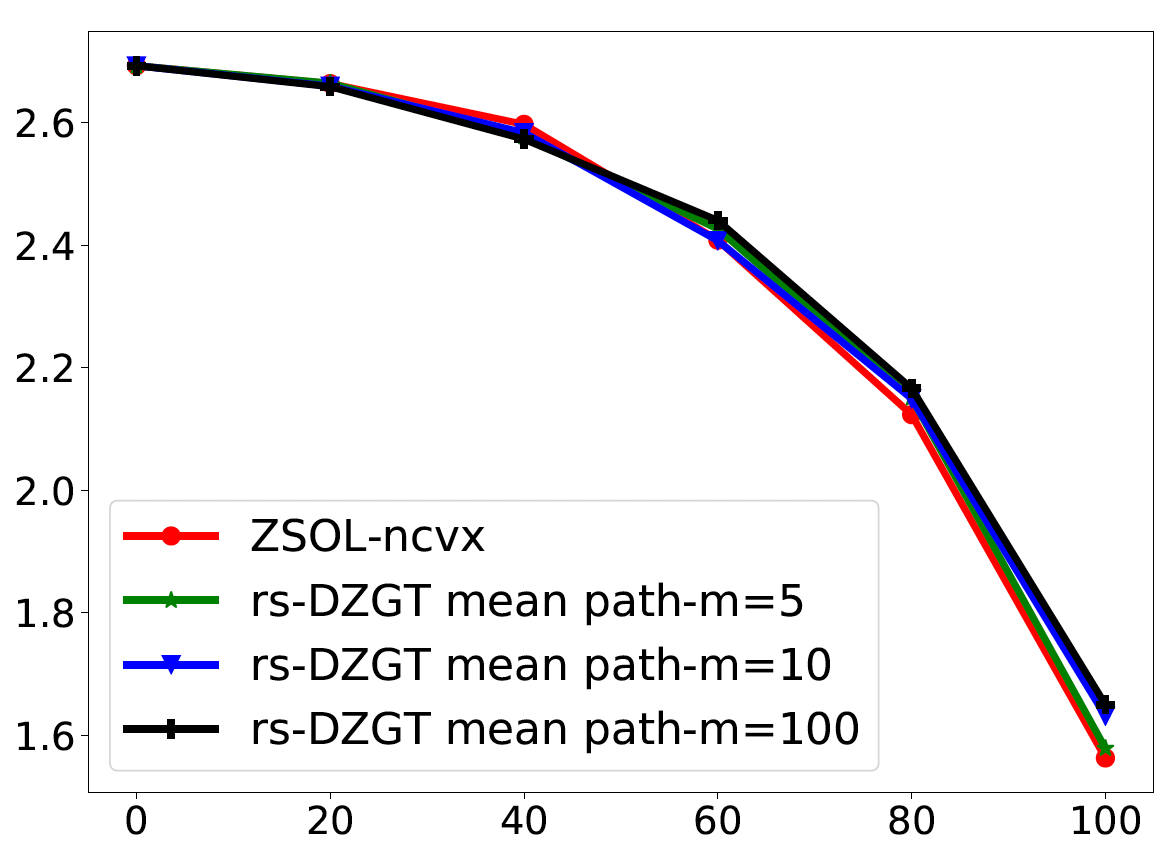}
\end{minipage}
\\ 
\hline\\
\rotatebox[origin=c]{90}{{\footnotesize {Sample \fy{ave.} implicit objective}}}
&
\begin{minipage}{.30\textwidth}
\includegraphics[scale=.199, angle=0]{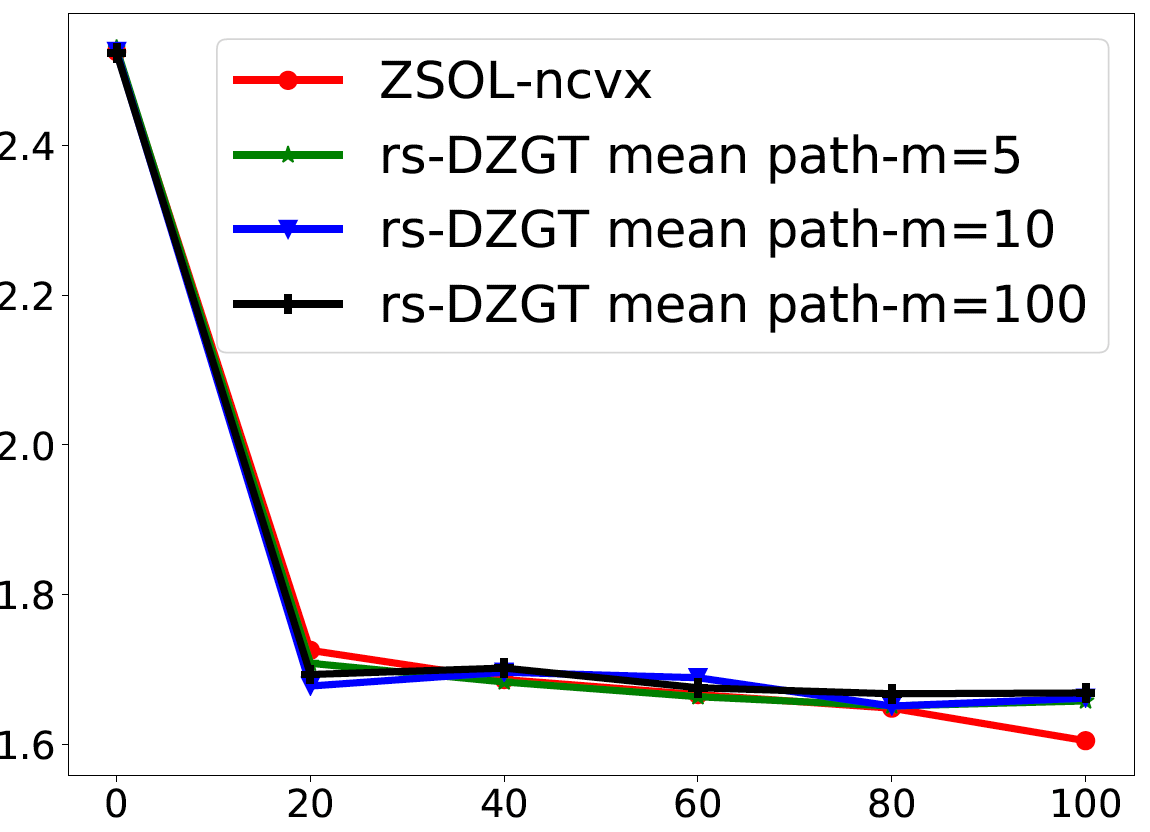}
\end{minipage}
&
\begin{minipage}{.30\textwidth}
\includegraphics[scale=.199, angle=0]{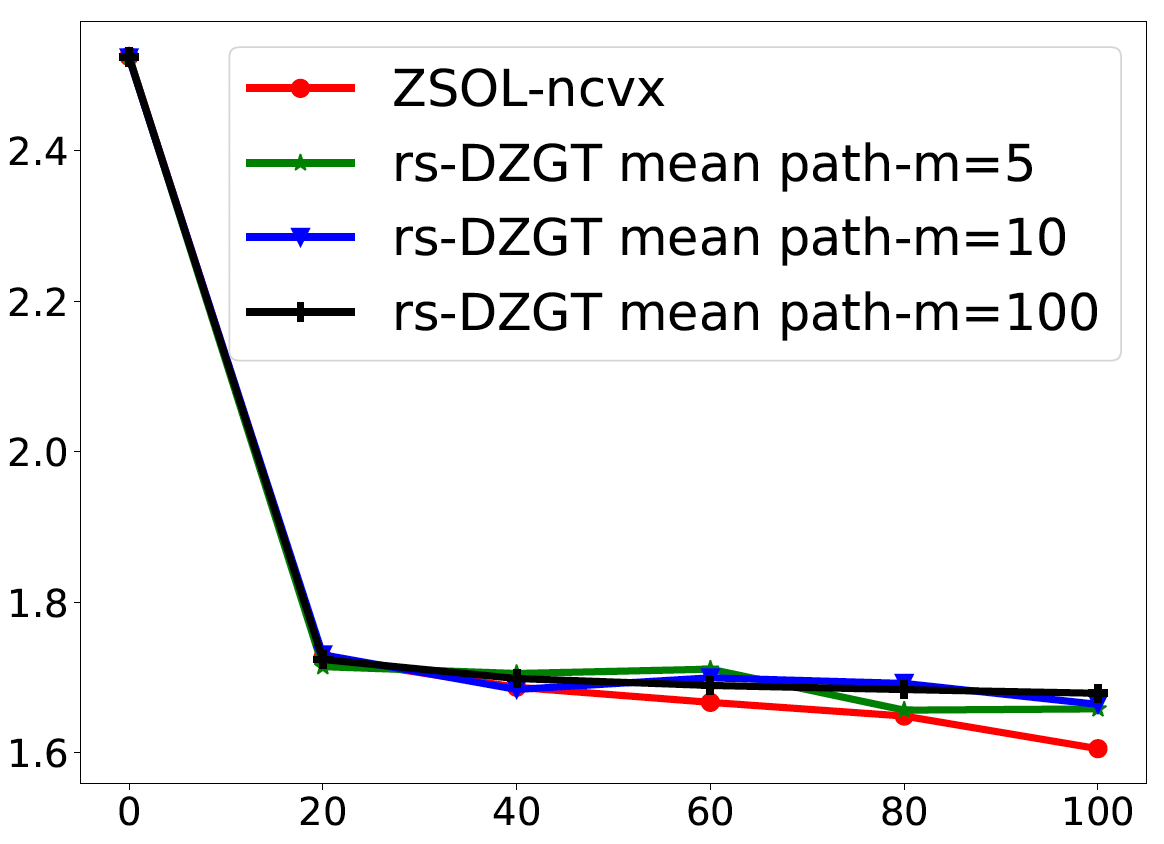}
\end{minipage}
	&
\begin{minipage}{.30\textwidth}
\includegraphics[scale=.199, angle=0]{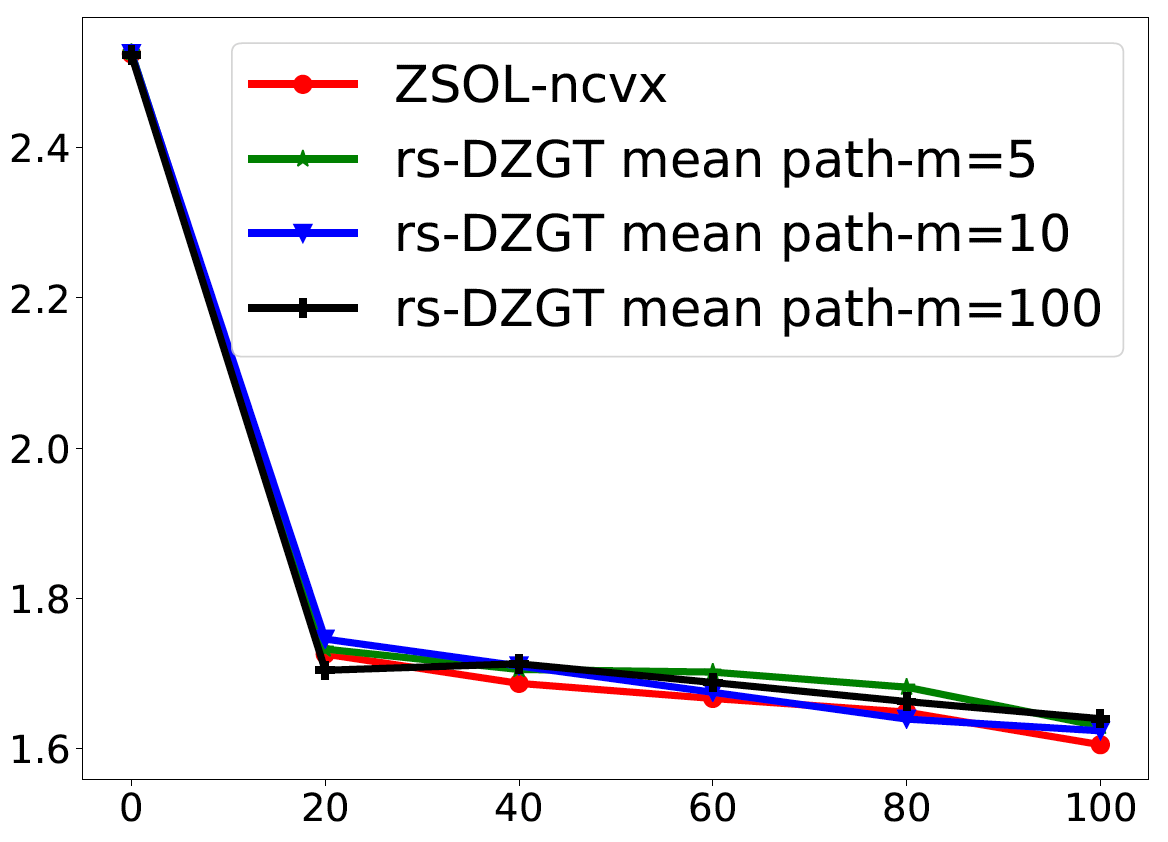}
\end{minipage}
\end{tabular}
\end{table} 
\captionof{figure}{Sample average implicit objective function computed by Algorithm $\ref{alg:DZGT}$ for various network sizes compared to sample average implicit objective function computed by ZSOL-ncvx. The \fy{stepsize} in the two rows are different.  }
\label{fig:comparison}}
\end{figure*}

\begin{table}[th]
\caption{Consensus \fy{error} for Algorithm~\ref{alg:DZGT} at the last epoch, under the first stepsize setting}
\label{table_example}
\begin{center}
\begin{tabular}{|c|c|c|c|}
\hline
Setting & Ring graph & Sparse graph & Complete graph\\
\hline
m=1 & 0 & 0 & 0\\
\hline
m=5 & 6.1266e-4 & 5.7364e-4 & 3.4662e-4\\
\hline
m=10 & 8.5629e-3 & 7.8391e-3 & 5.5593e-3\\
\hline
m=100 & 5.3591e-2 & 3.4981e-2 & 2.0635e-2\\
\hline
\end{tabular}
\end{center}
\end{table}

\begin{table}[h]
\caption{Consensus  \fy{error} for Algorithm~\ref{alg:DZGT} at the last epoch, under the second stepsize setting}
\label{table_example2}
\begin{center}
\begin{tabular}{|c|c|c|c|}
\hline
Setting & Ring graph & Sparse graph & Complete graph\\
\hline
m=1 & 0 & 0 & 0\\
\hline
m=5 & 3.6256e-3 & 2.9179e-3 & 2.3788e-3\\
\hline
m=10 & 2.0432e-2 & 1.5944e-2 & 1.2209e-2\\
\hline
m=100 & 6.1803e-1 & 4.4830e-1 & 3.4137e-2\\
\hline
\end{tabular}
\end{center}
\end{table}

\section{CONCLUSIONS}\label{sec:conc}
\fy{The mathematical} program with equilibrium constraint (MPEC) is a powerful model that captures several important problem classes such as Stackelberg games, bilevel optimization problems, and traffic equilibrium problems, to name a few. In this work, we consider stochastic variants of MPECs. Motivated by the absence of distributed \fy{schemes} for resolving this challenging mathematical model, we develop a novel gradient tracking method. Leveraging a randomized smoothing technique and inexact evaluations of the lower-level solutions, we develop a fully iterative distributed gradient tracking method. We derive complexity guarantees for computing a stationary point to the implicit optimization problem. We compare our method with its centralized counterpart and validate the theoretical guarantees over networks of different sizes and connectivity levels. Weakening the strong monotonicity assumption of the lower-level map is \fy{one interesting direction of} our future research. \fy{One possible avenue for addressing this appears to lie in employing iterative penalization (or regularization)~\cite{jalilzadeh2022stochastic,kaushik2021method}.}









\bibliographystyle{siam}

\bibliography{references_fy}

\end{document}